\documentclass{amsart}
\usepackage{amssymb,hyperref}
\usepackage{graphicx}
\usepackage[cmtip,all]{xy}

\newtheorem{theorem}{Theorem}[section]
\newtheorem{lemma}[theorem]{Lemma}
\newtheorem{proposition}[theorem]{Proposition}
\newtheorem{corollary}[theorem]{Corollary}

\theoremstyle{definition}
\newtheorem{definition}[theorem]{Definition}
\newtheorem{example}[theorem]{Example}

\theoremstyle{remark}
\newtheorem{remark}[theorem]{Remark}

\numberwithin{equation}{section}

\usepackage{stmaryrd}
\usepackage{gastex}

\newcommand{\Z}{{\mathbb Z}}     

\renewcommand{\geq}{\geqslant}
\renewcommand{\leq}{\leqslant}

\renewcommand{\preceq}{\preccurlyeq}
\renewcommand{\succeq}{\succcurlyeq}
\newcommand{\tvi}{\vrule height 12pt depth 5 pt width 0 pt}
\newcommand{\Rec}{\operatorname{Rec}}
\newcommand{\bs}{\backslash} 
\newcommand{\FVA}{F_{\mathcal V}(A)}
\newcommand{\FVB}{F_{\mathcal V}(B)}

\begin{document}

\title{Stone duality, topological algebra, and recognition}

\author{Mai Gehrke}
\address{LIAFA, University Paris-Diderot and CNRS, France.}
\email{mgehrke@liafa.univ-paris-diderot.fr}
\thanks{The author acknowledges support from
ANR 2010 BLAN 0202 02 FREC}


\begin{abstract}
Our main result is that any topological algebra based on a Boolean space is the extended Stone dual space of a certain associated Boolean algebra with additional operations. A particular case of this result is that the profinite completion of any abstract algebra is the extended Stone dual space of the Boolean algebra of recognizable subsets of the abstract algebra endowed with certain residuation operations. These results identify a connection between topological algebra as applied in algebra and Stone duality as applied in logic, and show that the notion of recognition originating in computer science is intrinsic to profinite completion in mathematics in general. This connection underlies a number of recent results in automata theory including a generalization of Eilenberg-Reiterman theory for regular languages and a new notion of compact recognition applying beyond the setting of regular languages. The purpose of this paper is to give the underlying duality theoretic result in its general form. Further we illustrate the power of the result by providing a few applications in topological algebra and language theory. In particular, we give a simple proof of the fact that any topological algebra quotient of a profinite algebra which is again based on a Boolean space is in fact a profinite algebra and we derive the conditions dual to the ones of the original Eilenberg theorem in a fully modular manner.
We cast our results in the setting of extended Priestley duality for distributive lattices with additional operations as some classes of languages of interest in automata theory fail to be closed under complementation.
\end{abstract}

\maketitle



\section{Introduction}\label{sec:intro}

In 1936, M. H. Stone initiated duality theory in logic by presenting a dual category equivalence between the category of Boolean algebras and the category of compact Hausdorff spaces having a basis of clopen sets, so-called Boolean spaces \cite{Stone36}. Stone's duality and its variants are central in making the link between syntactical and semantic approaches to logic. Also in theoretical computer science this link is central as the two sides correspond to specification languages and the space of computational states. This ability to translate faithfully between algebraic specification and spatial dynamics has often proved itself to be a powerful theoretical tool as well as a handle for making practical problems decidable. One may specifically mention Abramsky's paper \cite{Abramsky91} linking program logic and domain theory via Stone duality, Esakia's  duality \cite{Esa74} for Heyting algebras and the corresponding frame semantics for intuitionistic logic, and Goldblatt's paper \cite{Goldblatt89} identifying extended Stone duality as the setting for completeness issues for Kripke semantics in modal logic. These applications need more than just basic Stone duality as the first requires Stone or Priestley duality for distributive lattices and the latter two require a duality for Boolean algebras or distributive lattices endowed with additional operations. Dualities for additional operations originate with J\'onsson and Tarski \cite{JonssonTarski51,JonssonTarski52} and the first purely duality theoretic account may be found in \cite{Goldblatt89} in the setting of Priestley duality. Stone or Priestley duality for Boolean algebras and distributive lattices with various kinds of additional operations are often referred to as \emph{extended} duality.

Profinite algebra goes back at least to the paper \cite{B37} of Garrett Birkhoff, where he introduces topologies defined by congruences on abstract algebras observing that, if each congruence has finite index, then the completion of the topological algebra is compact. Profinite topologies for free groups were subsequently explored by M. Hall \cite{H50}. The profinite approach has also been used to much profit in semigroup theory and in automata theory since the late 1980s, in particular by Almeida, who developed the theory of so-called implicit operations \cite{Almeida05}. The abstract approach to formal languages and automata provided by profinite algebra has lead to the solution of very concrete problems in automata theory, like the filtration problem \cite{BB06} and the characterization of languages recognized by reversible automata \cite{P92}. 

Recognizability is an original subject of computer science. Relying on automata, the notion was first introduced for finite words by Kleene \cite{K56}, but was soon extended to infinite words by B\"uchi \cite{Bu62}, and then further to general algebras \cite{MW67}, finite and infinite trees \cite{D70,ThaWri68,Rab69}, and to many other structures. New settings in which recognition is a fruitful concept are still being developed, for example cost functions \cite{Col09} and data monoids \cite{Boj11}. The success of the concept of recognizability has been greatly augmented by its combination with profinite methods. 

Our main result is a link between topological algebras based on Boolean spaces and extended Stone duality, two distinct applications of topological methods in algebra. In particular, we show that topological algebras based on Boolean spaces are always themselves dual spaces of certain Boolean algebras with additional operations. The bulk of the paper studies this connection in detail, identifying the dual class of Boolean algebras with additional operations, the correspondence for morphisms, and the generalization to Priestley topological algebras and their distributive lattice with additional operations duals. In the special case of the profinite completion of an algebra of any operational type, the dual Boolean algebra with additional operations is the algebra of recognizable subsets of the original algebra endowed with certain operations. This result makes clear that the use in tandem of profinite completions and recognizable subsets in automata theory is not accidental. Since the two are duals of each other, the study of recognizable subsets is natural, not just in automata theory and theoretical computer science, but in any setting where profinite completions occur and vice versa.  The fact that the profinite completion of the free monoid on a finite set of generators is the dual space of a Boolean algebra with additional operations based on the recognizable subsets of the free monoid underlies a number of recent results in automata theory including a generalization of Eilenberg-Reiterman theory for regular languages \cite{GeGrPi08} and a notion of compact recognition applying beyond the setting of regular languages \cite{GehrkeGrigorieffPin10}. 

The paper is organized as follows. In Section~\ref{sec:duality} we provide the required preliminaries on duality theory. This material is not available in the needed form and with a uniform presentation, so we go in some detail. We include the discrete duality due essentially to Birkhoff as it underlies the topological one and is especially important for understanding additional operations. We describe the correspondences across the discrete and topological dualities for homomorphisms, subalgebras, and additional operations with some meet or join preservation properties. Section~\ref{sec:TopAlg} contains the main results of the paper. We show that topological algebras over Priestley spaces are dual spaces of certain distributive lattices with additional operations, and we identify the special features of the objects on either side of the duality. Finally we consider duality for maps. In particular, we identify the dual notion to one topological algebra over a Priestley space being an (ordered) topological algebra quotient of another. This gives rise to the notion of residuation ideal. Profinite algebras are particular topological algebras based on Boolean spaces. In Section~\ref{sec:ProfAlgAppl}  we identify the lattices with additional operations dual to profinite algebras and use this characterization to prove that Boolean-topological quotients of profinite algebras are again profinite. Then we specialize further and consider those profinite algebras which are profinite completions. In particular we show that the profinite completion of any discrete abstract algebra is the dual space of the Boolean algebra of recognizable subsets of the original abstract algebra equipped with certain residuation operations. Our proof of this result uses the general results of Section~\ref{sec:TopAlg} and is more conceptually transparent than the one used in \cite{GeGrPi08} (see also Lemma 1 of \cite{Gehrke09}).  Finally, we show how Eilenberg-Reiterman theory comes about from the duality between sublattices and quotient spaces applied in this setting.

Most of the results of this paper as well as their proofs were first discovered using an algebraic approach to duality for lattices with additional operations know as the theory of canonical extensions \cite{GehrkeJonsson04}. However, in order to make the paper accessible to researchers only familiar with duality theory in its topological form, we have chosen to present the results and their proofs without reference to canonical extensions. This has the drawback that it is less transparent how we arrived at the right notions and statements of results. For an outline of the canonical extension approach to this material, see \cite{Gehrke09}. 


\section{Preliminaries on duality}\label{sec:duality}

In this section we collect the basic facts about duality and extended duality that we will need. We assume all lattices to be distributive and bounded with the least element denoted by $0$ and greatest element by $1$.

\subsection{Discrete duality}\label{subsec:discdual}

The starting point of the representation theory of distributive lattices is the classical theorem of Birkhoff for finite distributive lattices. Also, duality for additional operations in the infinite topological setting is obtained by adding topological requirements to the underlying discrete duality. For this reason it is interesting to review here this discrete duality generalizing Birkhoff.

 An element $p$ in a lattice is called join-irreducible provided $p \neq 0$ and whenever $p = a \vee b$, we have $p = a$ or $p = b$.
\begin{theorem}[Birkhoff] \label{thm:birkhoff}
Any finite distributive lattice $D$ is isomorphic to the lattice of down-sets of the partially ordered set of join-irreducible elements of $D$ via the asignment for $a \in D$ 
\[
a\mapsto \widehat{a} := {\downarrow} a \cap J(D) = \{p \in D \ | \ p \text{ join-irreducible}, p \leq a\}.
\]
\end{theorem}

Birkhoff's duality generalizes to the category of complete lattices that are isomorphic to down-set lattices of posets. In the tradition of \cite{JonssonTarski51,GehrkeJonsson04}, we call these $DL^+$s. These lattices have a number of different abstract characterizations. They are the completely distributive complete lattices in which every element is the supremum of completely join-irreducible elements. Here, an element $p$ in a complete lattice $C$ is called completely join-irreducible provided, $p = \bigvee S$ with $S \subseteq C$ implies $p \in S$ and we denote the set of all completely join-irreducible elements of $C$ by $J^\infty(C)$. The $DL^+$s are also the doubly algebraic distributive lattices, see e.g. \cite[p. 83]{CrawleyDilworth73} for an early textbook source. Finally, this class of lattices was also rediscovered in the domain theory community where they are known as the prime algebraic distributive lattices \cite{NielsenPlotkinWinskel1981}. The Boolean members are the complete and atomic Boolean algebras, often denoted in the literature as CABAs or BA$^+$s.

\begin{theorem}\cite{Raney52}\label{cor:perfectbirkhoff}
Any $DL^+$ is isomorphic to the lattice of down-sets of the partially ordered set of its completely join-irreducible elements. In particular, a complete and atomic Boolean algebra is isomorphic to the powerset of its set of atoms.
\end{theorem}

This correspondence between  $DL^+$s and posets extends to a categorical duality in which complete lattice homomorphisms correspond to order-preserving maps. The correspondence between complete homomorphisms $h:C\to C'$ and order-preserving maps $\varphi:X'\to X$ is given by the following adjunction property for $x'\in X'=J^\infty(C')$ and $u\in C$ 
\[
\varphi(x')\leq u \quad\iff\quad x'\leq h(u).
\]
This works because $h$ has a lower adjoint which maps completely join-irreducibles to completely join-irreducibles and because $h$ may be recovered from this map. For further details, see Section 1.1 of \cite{Gehrke11}. 

 Let $C$ be a $DL^+$ and $X$ its poset of completely join irreducibles. Consider the following relation $R$ between elements $a\in C$ and pairs $(x,x')\in X\times X$:
\[
a \ R \ (x,x')\quad\iff\quad (x'\leq a\Rightarrow x\leq a).
\]
From this relation we get a Galois connection \cite{Birkhoff79} between the powersets of $C$ and $X\times X$ given by 
\begin{align*}
\mathcal E\ :\mathcal P(C) &\quad \leftrightarrows\quad \mathcal P(X\times X):\ \mathcal S\\
                                       K & \quad\mapsto\quad            \{(x,x')\mid\forall a\in K\ ( a \ R \ (x,x'))\}\\
\{a\mid\forall(x,x')\in \Delta \ (a \ R \ (x,x'))\}&\quad\mapsfrom \quad \Delta
\end{align*}

\begin{theorem}\label{thrm:dualityforcomplsubs}
Let $C$ be a $DL^+$ and $X$ its poset of completely join irreducibles. Further, let $\mathcal E\ :\mathcal P(C) \ \leftrightarrows\ \mathcal P(X\times X):\ \mathcal S$ be the above Galois connection. The Galois closed sets are the complete sublattices of $C$ and the quasi-orders on $X$ extending the partial order of $X$, respectively.
\end{theorem}

This is a poset theoretic generalization of the correspondence between complete Boolean subalgebras of power sets and equivalence relations on the underlying sets. In Section~\ref{PrDuality}  we will derive most of the corresponding result of topological duality (cf. Theorem~\ref{prop:subalg duality}).


An operation on a $DL^+$, $f:C^n\to C$, is a \emph{complete operator} provided it preserves arbitrary joins in each coordinate. For such an operation we have for each $\overline{u}\in C^n$
\[
f(\overline{u})=\bigvee\{f(\overline{x})\mid \overline{x}\in X^n\text{ with }\overline{x}\leq\overline{u}\}
\]
where $X$ is the poset of completely join irreducible elements of $C$. Define $R_f$ for $\overline{x}\in X^n$ and $x\in X$, by 
\[
\overline{x}R_fx\ \iff\ f(\overline{x})\geq x.
\] 
One may observe that the relations thus obtained are order-compatible in the following sense.

\begin{definition}
Let $X$ be a poset and $R\subseteq X^{n+1}$. We say that $R$ is \emph{order-compatible} provided for all $\overline{x},\overline{x}'\in X^n$ and all $x,x'\in X$, if $\overline{x}'\geq\overline{x}$ and $\overline{x}Rx$ and $x\geq x'$, then $\overline{x}'Rx'$.
\end{definition}

One obtains the following discrete duality theorem for complete operators.

\begin{theorem}\cite{JonssonTarski52}
Let $C$ be a $DL^+$ and $X$ its poset of completely join irreducibles. Discrete duality yields a one-to-one correspondence between the complete $n$-ary operators on $C$ and the order-compatible $(n+1)$-ary relations on $X$. It is given by
\begin{align*}
f:C^n\to C\qquad&\mapsto\qquad R_f=\{(\overline{x},x)\mid  x\leq f(\overline{x})\}\\
R\subseteq X^{n+1}\qquad &\mapsto\qquad f_R:{\mathcal D}(X)^n\to {\mathcal D}(X), \text{ with } f_R(\overline{U})=R[U_1,\ldots,U_n,\underline{\ }]
\end{align*}
where\\ 
$R[U_1,\ldots,U_n,\underline{\ }]=\{x\in X\mid \exists x_i\in U_i, i=1,\ldots,n,\text{ with }(x_1,\ldots,x_n,x)\in R\}$.
\end{theorem}

It is well known that an operation on a complete lattice, $f:C^n\to C$, is completely join-preserving in the $i$th coordinate if and only if it has an $i$th upper residual $f^\#_i:C^n\to C$. That is, $f$ and $f^\#_i$ are related by
\[
\forall a_1,\ldots,a_n,a\in C  \left(f(a_1,\ldots,a_n)\leq a\ \iff\ a_i\leq f^\#_i(a_1,\ldots,a_{i-1},a_{i+1},\ldots,a_n,a)\right).\footnote{For binary operations with infix notation, we denote the two upper residuals as right and left division, see e.g. Section~\ref{ExtDuality}.}
\]

Also, these two maps uniquely determine each other and the fact that $f^\#_i$ has a lower residual is equivalent to the fact that it turns arbitrary meets in the last coordinate into meets. If, in addition, $f$ is completely join preserving in each of its other coordinates, then $f^\#_i$ also turns arbitrary joins in each of the first $n-1$ coordinates into meets. The relation $R$ dual to a complete operator $f$ may also be seen as the dual of the upper residuals of $f$. The $i$th residual is given on the down-set lattice of $X$ by 
\begin{align*}
(U_1,\ldots,U_{i-1},U_{i+1},\ldots,U_n,U)&\mapsto \left(R[U_1,\ldots,U_{i-1},\underline{\ \ },U_{i+1},\ldots,U_n,U^c]\right)^c
\end{align*}
where $(\ \ )^c$ stands for the set-theoretic complement.

For more details on residuation see Section~4 of \cite{GehrkePriestleyPP2}. The binary case is also discussed further in Section~\ref{ExtDuality} below. 

\subsection{Ideals and filters}

The basic idea of lattice duality is to represent a lattice by its set of join- and/or meet-irreducible elements. However, for infinite lattices, there aren't necessarily enough of these, and idealized elements, in the form of ideals or filters, and topology must be considered.
	
Let $D$ be a bounded distributive lattice. A subset $I$ of $D$ is an ideal provided it is a down-set closed under finite joins. We denote by $Idl(D)$ the set of all ideals of $D$ partially ordered by inclusion. A subset $F$ of $D$ is a filter provided it is an up-set closed under finite meets. Filters represent (possibly non-existing) infima and thus the order on filters is given by \emph{reverse} inclusion. We denote by $Filt(D)$ the partially ordered set of all filters of $D$. A proper ideal $I$ is \emph{prime} provided $a\wedge b\in I$ implies $a\in I$ or $b\in I$.  A proper filter $F$ is \emph{prime} provided $a\vee b\in F$ implies $a\in F$ or $b\in F$. Note that a filter is prime if and only if its complement is an ideal, which is then necessarily prime, so that prime filters and prime ideals come in complementary pairs. In particular this means that the set of prime ideals with the inclusion order is isomorphic to the set of prime filters with the reverse inclusion order. For a bounded distributive lattice $D$ we will denote this partially ordered set by $X_D$ or just $X$. Since there are so many set theoretic levels in use when one talks about duality, we will revert to lower case letters $x,y,z\ldots$ for elements of $X$ and to make clear when we talk about the corresponding prime filter or the complementary prime ideal we will denote these by $F_x$ and $I_x$, respectively.  

\subsection{Stone and Priestley duality}\label{PrDuality}

For any bounded distributive lattice $D$ the following map is a bounded lattice homomorphism
\begin{align*}
\eta_D:D&\to{\mathcal P}(X_D)\\
       a&\mapsto \eta_D(a)=\{x\in X_D\mid a\in F_x\}
\end{align*}
Using the Axiom of Choice one may in addition show that $D$ has enough prime filters/ideals in the sense that this map also is injective. The Stone dual space \cite{Stone37} of $D$ is the topological space $(X_D,\sigma)$ where $\sigma$ is the topology on $X_D$ generated by the image of the map $\eta_D$, that is, by the basis
\[
\{\eta_D(a)\mid a\in D\}.
\]
For a Boolean algebra this yields a compact Hausdorff space for which the above basis is precisely the collection of clopen subsets of the space. For a non-Boolean bounded distributive lattice the corresponding Stone space is not $T_1$ separated and its specialization order is given by inclusion on the prime filters. The later Priestley variant of Stone duality \cite{Priestley70a} relies on the fact that every bounded distributive lattice, $D$, has a unique Boolean extension, $D^-$, whose prime filters are in one-to-one correspondence with the prime filters of $D$ and that may be obtained by generating a Boolean subalgebra of $\mathcal P(X_D)$ with the image of $\eta_D$. Thus the Priestley dual space of a bounded distributive lattice $D$ is the ordered topological space $(X_D,\leq,\pi)$ where 
\[
x\leq y \iff F_x\supseteq F_y \iff I_x\subseteq I_y
\]
and $\pi$ is the topology on $X_D$ generated by the subbasis
\[
\{\eta_D(a),(\eta_D(a))^c\mid a\in D\}.
\]
In the case where the lattice $D$ is a Boolean algebra, the Priestley duality agrees with the original Stone duality for Boolean algebras \cite{Stone36} and we may refer to it as Stone duality rather than as Priestley duality. The dual of a homomorphism $h:D\to E$ between distributive lattices in Priestley duality (as well as in Stone duality) is the map $f:X_E\to X_D$ such that $f(x)=y$ if and only if $h^{-1}(F_x)=F_y$. One can then show that the space $(X_D,\leq,\pi)$ is compact and \emph{totally order disconnected}, that is, for $x,y\in X_D$ with $x\nleq y$ there is a clopen down-set $U$ with $y\in U$ and $x\not\in U$. Also, for any homomorphism $h:D\to E$, the map $h^{-1}:X_E\to X_D$ is continuous and order preserving. 

 A \emph{Priestley space} is an ordered topological space that is compact and totally order disconnected and the morphisms of Priestley spaces are the order preserving continuous maps. The dual of a Priestley space $(X,\leq,\pi)$ is the bounded distributive lattice $ClopD(X,\leq,\pi)$ of all subsets of $X$ that are simultaneously clopen and are down-sets. For $\varphi:X\to Y$ a morphism of Priestley spaces, the restriction of the inverse image map to clopen down-sets, $\varphi^{-1}: ClopD(Y)\to ClopD(X)$, is a bounded lattice homomorphism and is the dual of $\varphi$ under Priestley duality. The translations back and forth given above account for Priestley duality. It allows one to translate essentially all structure, concepts, and problems back and forth between the two sides of the duality. 

One particular case of this translation across the duality is the correspondence between bounded sublattices of a lattice and the Priestley quotients of the dual space of the lattice. This is central to this work, and, while it is well known to duality theorists, we will supply some details here. 

Let $i:A\hookrightarrow B$ be an inclusion of bounded distributive lattices. Its dual is a quotient map $X_B\ \twoheadrightarrow \ X_A$ where $x\in X_B$ is sent to the point of the dual of $A$ corresponding to the prime filter $i^{-1}(F_x)=F_x\cap A$. That is, in terms of prime filters, the quotient map is given by restricting the prime filters of $B$ to $A$. The kernel of this quotient map is a quasiorder containing the 
order on $X_B$. One can characterize the quasiorders arising in this way and this describes the correspondence. However, we can get something a bit better, namely, a Galois connection whose Galois closed sets are the bounded sublattices on one side and the appropriate quasiorders on the other.

Let $B$ be bounded distributive lattice and $S$ a subset of $B$. Then $S$ gives rise to a binary relation on $X_B$ given by
\[
x\preceq_S y \ \iff\ \forall\ a\in S\quad (a\in F_y\ \Rightarrow\ a\in F_x).
\]
It is easy to verify that $\preceq_S$ is a quasiorder extending the order on $X_B$. In the other direction, given a subset $E\subseteq X_B\times X_B$, we obtain a subset $A_E$ of $B$ given by
\[
A_{E}=\{a\in B\mid \forall(x,y)\in E\quad(a\in F_y\ \Rightarrow\ a\in F_x)\}.
\]
Here again it is easy to show that, for any $E\subseteq X_B\times X_B$, the set $A_E$ is a bounded sublattice of $B$. The key facts are the following.

\begin{proposition}\label{prop:quasiorder-down-sets}
Let $B$ be a bounded distributive lattice and let $A$ be a bounded sublattice of $B$. Then we have 
\[
A_{\preceq_A}=A.
\]
\end{proposition} 

\begin{proof}
Let $a_0\in A$ and suppose $x\preceq_A y$, that is, $(x,y)\in \preceq_A$. Then by definition of $\preceq_A$, if  $a_0\in F_y$ it follows that $a_0\in F_x$ and thus $a_0\in A_{\preceq_A}$.  Conversely, let $b\in A_{\preceq_A}$. Fix $x\in X_B$ with $b\not\in F_x$. For each $y\in X_B$ with $b\in F_y$ we then must have $x\not\preceq_A y$ since $b\in A_{\preceq_A}$. Thus there is $a_y\in A$ with $a_y\in F_y$ but $a_y\not\in F_x$. Now we have
\[
\eta_B (b)=\{y\in X_B\mid b\in F_y\}\subseteq\bigcup\{\eta_B (a_y)\mid y\in X_B\text{ and } b\in F_y\}.
\]
By compactness of $\eta_B (b)$, it follows that there are $y_1,\ldots,y_n\in X_B$ with $\eta_B (b)\subseteq\bigcup_{i=1}^{n}\eta_B (a_{y_i})$. Let $a_x=\bigvee_{i=1}^{n}a_{y_i}$, then the following are true: $b\leq a_x$ since $\eta_B$ is a lattice embedding and $a_x\in A$ since each of the $a_y$s are in $A$ and $A$ is closed under finite joins. Also $a_x\not\in F_x$ since $F_x$ is prime and $a_y\not\in F_x$ for each $y$. 

So for each $x\in X_B$ with $b\not\in F_x$, we have $a_x\in A$ with $b\leq a_x$ and $a_x\not\in F_x$. The two latter facts correspond to 
$x\in (\eta_B (a_x))^c\subseteq(\eta_B (b))^c$. Thus we have
\[
(\eta_B (b))^c=\bigcup\{(\eta_B (a_x))^c\mid x\in X_B\text{ and } b\not\in F_x\}.
\]
Again, by compactness, there must be $x_1,\ldots,x_m\in X_B$ with $b\not\in F_{x_j}$ for each $j$ and 
\[
(\eta_B (b))^c=\bigcup_{j=1}^m(\eta_B (a_{x_j}))^c.
\]
That is, $b=\bigwedge_{j=1}^m a_{x_j}$ and thus $b\in A$ since $A$ is closed under finite meets and each $a_x\in A$.
\end{proof}

Further, it is easy to see that the quasiorders of the form $\preceq_A$ have the following characteristic property which we call compatibility.

\begin{definition}\label{def:compatible}
Let $B$ be a bounded distributive lattice, $X_B$ the dual space of $B$. A quasiorder $\preceq$ on $X_B$ is said to be \emph{compatible} provided it satisfies
\[
\forall x,y\in X_B\quad [x\not\preceq y\ \Rightarrow \exists a\in B\ (a\in F_y
\mbox{ and }a\not\in F_x\mbox{ and }\eta_B(a)\mbox{ is a $\preceq$-down-set})].
\]
\end{definition}

It is straight forward to show that $\preceq_{A_\preceq}=\preceq$ for any compatible quasiorder $\preceq$ on the dual of a bounded distributive lattice as compatibility easily implies that the corresponding quotient space is a Priestley space. Note that the assignments $E\mapsto A_E$ and $S\mapsto \preceq_S$ are both derived from the relation
\[
(x,y)\ R\ a \quad\text{ defined by }\quad a\in F_y\ \Rightarrow\ a\in F_x.
\]
and thus they form a Galois connection. To sum up we have the following result.
 
\begin{theorem}[\cite{Schmid02}]\label{prop:subalg duality}
Let $B$ be a bounded distributive lattice, $X_B$ the dual space of $B$. The assignments
\[
E\mapsto A_E=\{a\in B\mid\forall(x,y)\in E\ (a\in F_y\ \Rightarrow\ a\in F_x)\}
\]
for $E\subseteq X_B\times X_B$ and
\[
S\mapsto \preceq_S=\{(x,y)\in X_B\times X_B\mid\forall a\in S\ (a\in F_y\ \Rightarrow\ a\in F_x)\}
\]
for $S\subseteq B$ establish a Galois connection whose Galois closed sets are the compatible quasiorders and the bounded sublattices, respectively.
\end{theorem}

We note that, throughout, the special case of Stone duality for Boolean algebras corresponds to the case where the order is trivial. 

\begin{remark}\label{rem:Priestleyspace}
We also note that the Priestley space of a distributive lattice is actually the dual space of the free Boolean extension $D^-$ of $D$ equipped with the compatible (quasi)order (which happens to be a partial order in this case, see \cite[Proposition 8]{Gehrke13}) dual to the sublattice inclusion map $D\hookrightarrow D^-$. For more details on this, see  \cite[Theorem 5]{Gehrke13}.
\end{remark}

\subsection{Extended Priestley duality}\label{ExtDuality}

In extended Priestley duality \cite{Goldblatt89}, additional operations on a distributive lattice are captured by additional relational structure on the dual space, see also \cite{GehrkePriestley07a,GehrkePriestleyPP2}. Here we give a brief description of the relational dual of the additional operations we will be most concerned with. We illustrate with a binary operation but corresponding results hold for operations of any arity. It is easiest to start with an operation $\cdot:D\times D\to D$ preserving finite (including empty) joins in each coordinate. If $D$ is finite, as in the discrete duality setting, it is enough to know the operation on pairs of join-irreducible elements. In the setting of arbitrary bounded distributive lattices this corresponds to knowing the action of the operation on the prime filters. For this purpose we extend the operation to an operation on the filter lattice in the obvious way:
\begin{align*}
\cdot_{Filt}: Filt(D)\times Filt(D) & \to Filt(D)\\
                          (F,G)                  &\mapsto \langle F\cdot G\rangle_{Filt}
\end{align*} 
where 
$\langle F\cdot G\rangle_{Filt}
={\uparrow}\{\bigwedge_{i=1}^n(a_i\cdot b_i)\mid a_i\in F\mbox{ and }
b_i\in G, i=1,\ldots,n\}$ 
is the filter generated by the product of $F$ and $G$. The operation on filters will not in general map pairs of prime filters to prime filters but the restriction of the operation to pairs of prime filters may be encoded by the relation
\begin{align*}
R_{\,\cdot} &=\{(x,y,z)\in (X_D)^3\mid F_x\cdot_{Filt} F_y\geq F_z\}\\
                 &=\{(x,y,z)\in (X_D)^3\mid F_x\cdot F_y\subseteq F_z\}.
\end{align*} 
In the case where the original operation preserves finite joins in each coordinate one can show that one recovers the original operation as
\begin{align*}
ClopD(X_D)^2
\quad &\longrightarrow\quad ClopD(X_D)\\
             (U,V)\qquad        \quad&\mapsto\qquad R_{\,\cdot}[U,V,\underline{\ }]
                                                                   =\{z\in X_D\mid\exists x\in U,y\in V
                                                                   \ R_{\,\cdot}(x,y,z)\}.
\end{align*}
Further, it may be shown that the relations $R$ corresponding to binary operations that preserve finite joins in each coordinate are the ones satisfying the following three properties \cite{Goldblatt89}: (Notice that our last coordinate is the first coordinate in \cite{Goldblatt89})
\begin{enumerate}
	   \item $(\geq \times \geq) \circ R \circ {\geq} = R$;
	\item For each $x\in X$ the set $R[\_,\_,x]$ is closed;
	   \item For all $U,V$ clopen down-sets of $X$ the set $R[U,V,\_]$ is clopen.
\end{enumerate} 

For operations with other preservation properties one has to apply some order duality (that is, turn the lattice upside-down). For this to work it is important that all domain coordinates transform to joins in the codomain or all transform to meets. For example, for an operation $\bs:D\times D\to D$ that sends finite joins in the first coordinate to finite meets and finite meets in the second coordinate to finite meets (when one fixes the other coordinate), we must first extend $\bs$ to a function from $Filt(D)\times Idl(D)$ into $Idl(D)$ by setting
\[
	F\bs I= \langle\, a\bs b\mid a\in F\text{ and }b\in I\,\rangle_{Idl}
\]
for $F\in Filt(D)$ and $I\in Idl(D)$. The relation dual to $\bs$ is then
\begin{align*}
S_{\,\bs} & =\{(x,y,z)\in X^3\mid  F_x\bs I_z\leq I_y\}\\
              & =\{(x,y,z)\in X^3\mid  F_x\bs I_z\subseteq I_y\}.
\end{align*}
and the original operation $\bs$ is captured on clopen down-sets by 
\[
U\bs V=(S_{\,\bs}[U,\_,V^c])^c=\{y \mid \forall x,z\ [(x\in U\text{ and }S_{\,\bs}(x,y,z))\implies z\in V]\}.
\]
Furthermore, a relation $S$ is the dual of some operation $\bs$ which sends finite joins in the first coordinate to finite meets and finite meets in the second coordinate to finite meets if and only if it satisfies the following three properties:
\begin{enumerate}
	\item $(\geq \times \geq) \circ S \circ {\geq} = S$;
	\item For each $x\in X$ the set $S[\_,x,\_]$ is closed;
	\item For all $U$ clopen down-set of $X$ and $V$ clopen up-set of $X$, the set $S[U,\_,V]$ is clopen.
\end{enumerate}

In the sequel we will be applying these results in a situation where we have a family of operations $(\cdot, \bs,/)$\footnote{Sometimes we won't have all the operations of the residuated family available on the lattice.} that form a residuated family on a bounded distributive lattice $D$. That is, for all $a,b,c\in D$ we have 
\begin{align*}
	a\cdot b\leq c &\iff b\leq a\bs c\\
	                         & \iff a\leq c/b.
\end{align*}
In this case one can prove that all three operations are encoded on the dual space by a single relation $R$ which may be defined by any of the following equivalent conditions 
\begin{align*}
R(x,y,z) \iff F_x\cdot F_y\geq F_z
                         \iff F_x\bs I_z\leq I_y
                         \iff I_z/ F_y\leq I_x.
\end{align*}
Conversely, given a ternary relation $R$ on a Priestley space $X$, which is order-compatible so that it satisfies 
\[
(\geq \times \geq) \circ R \circ {\geq} = R,
\] 
we obtain, via discrete duality, a residuated family of maps on the lattice of down-sets of $X$ given by 
\begin{align*}
	S\cdot T & = R[S,T,\_] =\{z \mid \exists x,y\ [x\in S\text{ and } y\in T \text{ and } R(x,y,z)]\}\\ 
        S\bs T & =(R[S,\_,T^c])^c =\{y \mid \forall x,z\ [(x\in S\text{ and }R(x,y,z))\implies z\in T]\}\\
	T/S  & =(R[\_,S,T^c])^c =\{x\mid \forall y,z\ [(y\in S\text{ and } R(x,y,z))\implies z\in T]\}.
\end{align*}
However, the lattice of clopen down-sets may not be closed under some of these while being closed under others. In particular, a relation $R$ can be the topological dual for one of these operations while not being so for another one. This is determined by the topological properties of the relation $R$. As stated above, this relation $R$ is dual to an operation $\cdot$ on $D$  with the third coordinate as output variable if and only if
\begin{enumerate}
	  \item For each $x\in X$ the set $R[\_,\_,x]$ is closed;
	   \item For all $U,V$ clopen down-sets of $X$ the set $R[U,V,\_]$ is clopen.
\end{enumerate}
When this is the case we say that $R$ is Priestley-compatible for the last coordinate.

We state the topological properties for the residual operations in a definition as these are particularly central in this work and we will want to refer to them later.

\begin{definition} \label{def:topprores}
Let $X$ be a Priestley space and $R\subseteq X^n\times X$ an order-compatible relation on $X$. For $1\leq i\leq n$, we say that $R$ is Priestley-compatible with i as the output coordinate provided:
\begin{enumerate}
	\item For each $x\in X$ the set $R[\_,x,\_]$, where $x$ occurs in the $i$th coordinate, is closed;
	\item For all $U_1,\ldots, U_{i-1}$ and $V_{i+1}\ldots,V_n$ clopen down-set of $X$ and for all $W$ clopen up-set of $X$, the set $R[\overline{U},\_,\overline{V},W]$ is clopen.
\end{enumerate}
\end{definition}

In the setting of lattices with additional operations, we want homomorphisms to preserve both the lattice structure and the additional operations. The dual notion is known under the name of bounded morphism, see e.g. \cite{Goldblatt89}. This is the functional version of bisimulation in modal logic. 

\begin{definition}\label{def:bddmorphism}
Let $X$ and $Y$ be Priestley spaces, $R\subseteq X^{3}$ and $S\subseteq Y^{3}$ order-compatible relations on $X$ and $Y$, respectively. If $R$ and $S$ are Priestley-compatible with respect to the last coordinate, then we say that a continuous and order-preserving function $\varphi:X\to Y$ is a bounded morphism for these relations with respect to the last coordinate if and only if the following two properties, known as the \emph{Back} and \emph{Forth} properties, hold for all $x_1,x_2,x_3\in X$ and all $y_1,y_2\in Y$ 
\begin{itemize}
\item[]\hspace*{-.75cm}(Forth)\ $(R(x_1,x_2,x_3)\ \Rightarrow\  S(\varphi(x_1),\varphi(x_2),\varphi(x_3))\,)$;
\item[]\hspace*{-.75cm}(Back)\ $(S(y_1,y_2,\varphi(x_3))\ \Rightarrow \ 
                                                                \exists x_1,x_2\ [\,(y_1,y_2)\geq(\varphi(x_1),\varphi(x_2))\text{ and }R(x_1,x_2,x_3)\,]\,).$
\end{itemize}

Similarly, if $R\subseteq X^{3}$ and $S\subseteq Y^{3}$ are Priestley-compatible with respect to the second coordinate (that is, they are duals of operations of the type $\bs$), then we say that a continuous and order-preserving function $\varphi:X\to Y$ is a bounded morphism for these relations with respect to the second coordinate if and only if the following two properties hold for all $x_1,x_2,x_3\in X$ and all $y_1,y_3\in Y$ 
\begin{itemize}
\item[]\hspace*{-.75cm}(Forth)\ $(R(x_1,x_2,x_3)\ \Rightarrow\  S(\varphi(x_1),\varphi(x_2),\varphi(x_3))\,)$;
\item[]\hspace*{-.75cm}(Back)\ $(S(y_1,\varphi(x_2),y_3) \Rightarrow 
                                                      \exists x_1,x_3\ [\,y_1\geq \varphi(x_1), R(x_1,x_2,x_3),\text{ and }\varphi(x_3)\geq y_3]).$
\end{itemize}

If $R\subseteq X^{3}$ and $S\subseteq Y^{3}$ are Priestley-compatible with respect to the first coordinate (that is, they are duals of operations of the type $/$), then we say that a continuous and order-preserving function $\varphi:X\to Y$ is a bounded morphism for these relations with respect to the first coordinate if and only if the following two properties hold for all $x_1,x_2,x_3\in X$ and all $y_2,y_3\in Y$ 

\begin{itemize}
\item[]\hspace*{-.75cm}(Forth)\ $(R(x_1,x_2,x_3)\ \Rightarrow\  S(\varphi(x_1),\varphi(x_2),\varphi(x_3))\,)$;
\item[]\hspace*{-.75cm}(Back)\ $(S(\varphi(x_1),y_2,y_3) \Rightarrow  
                                                      \exists x_2,x_3\ [y_2\geq \varphi(x_2),R(x_1,x_2,x_3),\text{ and }\varphi(x_3)\geq y_3]).$
\end{itemize}

Finally, we note that in the special case where $\varphi$ is surjective, and thus corresponds to a compatible quasiorder $\preceq$ on $X$, if the quotient map is a bounded morphism for relations $R$ on $X$ and $S$ on the quotient Priestley space $(X/{\equiv},{\preceq}/{\equiv}, \pi/{\equiv})$ with respect to any one of their coordinates, then $S$ is the quotient relation, $R/{\preceq}$, in the sense that  for all $(\overline{x},x)\in X^n\times X$ we have 
\[
([x_1],\ldots,[x_{n}])\, S\, [x] \ \iff\ \overline{x}\, [\succeq^n\circ\,R\,\circ\succeq]\, x.
\] 
\end{definition}

\bigskip
\bigskip
%

\section{Topological algebras as dual spaces}\label{sec:TopAlg}

In this section we study the relationship between extended dual spaces and topological algebras. As we have seen in the previous section, an extended dual space is a Boolean space, or a Priestley space, with additional relations having some topological, and, in the case of Priestley spaces, order-theoretic properties. Boolean-topological algebras are Boolean spaces equipped with continuous operations. Our main result in this section is that Boolean-topological algebras are precisely the extended dual spaces for which the additional relations are functional. Thus Boolean-topological algebras are special extended dual spaces. In the Priestley setting the relationship is a bit more complicated: All Priestley topological algebras are extended dual spaces but they do not comprise all the functional ones.

Once we have established these results, we characterize those distributive lattices with additional operations whose extended dual spaces have functional relations. Finally we identify the lattice theoretic duals of morphisms and quotients of Boolean and Priestley topological algebras. In particular we show that the dual of a Boolean-topological quotient is what we will call a residuation ideal. 

\subsection{Functional dual relations}		

Let $X$ be a Priestley space and $R$ an $(n+1)$-ary relation on $X$ that is order-compatible. As we have seen in the previous section, $R$ corresponds via discrete duality to a residuated family of $n$-ary operations on the $DL^+$, $\mathcal D(X)$, consisting of all down-sets of $X$. Depending on the topological properties of this relation, it may then be the dual of any number of these operations restricted to the dual lattice. In any case, order compatibility makes it impossible for the relation to be functional unless the order is trivial (the Boolean case) or the value of the function is always minimal in $X$. In order to encompass a richer class of structures we give the following relaxed definition of functionality for relations on Priestley spaces.

\begin{definition}\label{def:functrel}
Let $X$ be a Priestley space and $R\subseteq X^n\times X$ an $(n+1)$-ary relation on $X$. We say that $R$ is \emph{functional} provided there is an $n$-ary operation, $f$, on $X$ such that for all $\overline{x}\in X^n$ and all $z\in X$, we have $R(\overline{x},z)$ if and only if $f(\overline{x})\geq z$. 
\end{definition}

Note that if $R$ is functional then the corresponding operation on $X$ is uniquely given by $f(\overline{x})=\max\{z\in X\mid R(\overline{x},z)\}$. The truth of the following proposition is easy to verify.

\begin{proposition}\label{prop:order-comp}
Let $X$ be a Priestley space and $R$ a functional relation on $X$ with $f$ the corresponding operation. Then $f$ is order preserving if and only if $R$ is order-compatible. 
\end{proposition}

The following establishes a link between continuity of an operation on a Priestley space and the corresponding functional relation being the topological dual of the residual operations given by the relation.

\begin{proposition}\label{prop:cont=Kripke}
Let $X$ be a Priestley space and $R\subseteq X^n\times X$ an order-compatible functional relation on $X$ with $f$ the corresponding operation. Then the following conditions are related by {\rm(1)}$\,\Rightarrow${\rm(2)}, {\rm (2)}$\,\Rightarrow${\rm(3)}, and {\rm (3)}$\,\Rightarrow${\rm(4)}.
\begin{enumerate}
	\item The operation $f$ is continuous with respect to the Priestley topology.
	\item For each $i$ with $1\leq i\leq n$ and all $x \in X$, $R[\underline{\ \ \ },x,\underline{\ \ \ }]$ (where $x$ is in the $i$th spot) is closed in $X^n$ and for all clopen down-sets $U_j,V\subseteq X$ the relational image $R[U_1,\ldots,U_{i-1},\_,U_{i+1},\ldots,V^c]$ is clopen.
	\item There is an $i$ with $1\leq i\leq n$ such that for all $x \in X$, $R[\underline{\ \ \ },x,\underline{\ \ \ }]$ (where $x$ is in the $i$th spot) is closed in $X^n$ and for all clopen down-sets $U_j,V\subseteq X$ the relational image $R[U_1,\ldots,U_{i-1},\_,U_{i+1},\ldots,V^c]$ is clopen.
        \item The operation $f$ is continuous with respect to the spectral topology.
\end{enumerate}
\end{proposition}

\begin{proof} 
Assuming that (1) holds, we just prove (2) for $i=1$ to minimize notation. If $f$ is continuous in the Priestley topology, then, for each $x_1\in X$, the function $f_{x_1}:X^{n-1}\to X$ given by $\overline{y}\mapsto f(x_1,\overline{y})$ is continuous in the Priestley topology and thus its graph, $G(f_{x_1})$, is closed in $X^n$. Also notice that if $X^{\rm op}$ denotes the Priestley space obtained by reversing the order of $X$, then $R[x_1,\underline{\ \ \ }]$ is the down-set of $G(f_{x_1})$ in the order of the space $(X^{\rm op})^{n-1}\times X$. Now using the fact that the product of Priestley spaces is a Priestley space, and that down-sets of closed sets are closed in Priestley spaces \cite[Exercise11.14(ii)]{DaveyPriestley02}, we conclude that $R[x_1,\underline{\ \ \ }]$ is closed. 

Let $U_2,\ldots, U_n,V$ be clopen down-sets in $X$. By continuity of $f$, the set $f^{-1}(V^c) $ is clopen. Also, because $V^c$ is an up-set,  $f^{-1}(V^c) = R[\underline{\ \ \ },V^c]$. Now let $\pi:X\times X^{n-1}\to X$ be the projection onto the first coordinate, then
\vspace*{-.1cm}
\[
	R[\_,U_2,\ldots, U_n,V^c]=\pi(R[\underline{\ \ \ },V^c]\cap(X\times U_2\times\ldots\times U_n)).
\]
\noindent
Since the intersection of clopen sets is clopen, projections of open sets are open, and projections of closed sets along compact Hausdorff spaces are closed, it follows that $R[\_,U_2,\ldots, U_n,V^c]$ is clopen. This completes the proof of (1) implies (2).

In order to prove that (3) implies (4), we assume (3) holds for $i=1$ and prove that $f$ is continuous with respect to the spectral topology. To this end, let $(x_1,\ldots,x_n)\in X^n$ with $z=f(x_1,\ldots,x_n)\in V$, where $V$ is a clopen down-set in $X$. Since $R[x_1,\underline{\ \ \ }]$ is closed and $V$ is clopen, it follows that $R[x_1,\underline{\ \ \ }]\cap(X^n\times V^c)$ is closed. Furthermore, since the projection $\pi':X^{n-1}\times X\to X^{n-1}$ is a projection along a compact space, it is a closed map and thus 
\[
  \pi'(R[x_1,\underline{\ \ \ }]\cap(X^{n-1}\times V^c))
\]
is closed. Notice that since $V^c$ is an up-set we have 
\begin{align*}
  \pi'(R[x_1,\underline{\ \ \ }]\cap(X^{n-1}\times V^c))
  &= \{(y_2,\ldots,y_n)\in X^{n-1}\mid\exists z\in V^c\ f(x_1,y_2,\ldots,y_n)\geq z\}\\
  &=  \{(y_2,\ldots,y_n)\in X^{n-1}\mid f(x_1,y_2,\ldots,y_n)\in V^c\}\\
  &=f_{x_1}^{-1}(V^c)
\end{align*}
where $f_{x_1}:X^{n-1}\to X$ is is the restriction of $f$ as defined above. It thus follows that $f_{x_1}^{-1}(V)$ is open. Since it is also a down-set and $\overline{x}=(x_2,\ldots,x_n)\in f_{x_1}^{-1}(V)$, there are clopen down-sets $U_2,\ldots, U_n$ with $\overline{x}\in U_2\times\ldots\times U_n\subseteq  f_{x_1}^{-1}(V)$. We have $x_1\not\in R[\_,U_2,\ldots, U_n,V^c]$, and thus $x_1\in (R[\_,U_2,\ldots, U_n,V^c])^c=U_1$ which is a 
down-set and is in addition clopen by condition (3). That is, $(x_1,\ldots,x_n)\in U_1\times U_2\times\ldots\times U_n$ and each $U_i$ is a clopen down-set. Furthermore, if $(y_1,\ldots,y_n)\in U_1\times U_2\times\ldots\times U_n$ then $y_1\in U_1=(R[\_,U_2,\ldots, U_n,V^c])^c$ 
and thus 
\vspace*{-.1cm}
\begin{align*}
y_1\not\in\  & R[\_,U_2,\ldots, U_n,V^c]\\
                   & =\{y'\mid\exists(y'_2,\ldots,y'_n)\in U_2\times\ldots\times U_n\ f(y',y'_2,\ldots,y'_n)\not\in V\}.
\end{align*}
\noindent
That is, for all $(y'_2,\ldots,y'_n)\in U_2\times\ldots\times U_n$ we have $f(y_1,y'_2,\ldots,y'_n)\in V$ and in particular $f(y_1,y_2,\ldots,y_n)\in V$. We have shown then that $(x_1,\ldots,x_n)\in U_1\times U_2\times\ldots\times U_n\subseteq f^{-1}(V)$ and thus that $f$ is continuous in the spectral topology.
\end{proof}

In the Boolean case, we obtain a stronger result since the Priestley and the spectral topologies are one and the same so that conditions (1) and (4) are equivalent. 

\begin{corollary}\label{cor:cont=Kripke}
Let $X$ be a Boolean space and let $f$ be an $n$-ary operation on $X$ and suppose $R\subseteq X^n\times X$ 
is the graph of $f$. Then the following conditions are equivalent: 
\begin{enumerate}
	\item There is an $i$ with $1\leq i\leq n$ such that $R$ is the extended Stone dual of the operation 
	$(U_1,\ldots,U_n)\mapsto R[U_1,\ldots,\underline{\ \ },\ldots,U_n]$ (with the co-domain slot in the $i$th place) 
	on the dual Boolean algebra.     
	\item For each $i$ with $1\leq i\leq n$, the relation $R$ is the extended Stone dual of the operation 
	$(U_1,\ldots,U_n)\mapsto R[U_1,\ldots,\underline{\ \ },\ldots,U_n]$  (with the co-domain slot in the $i$th place)
	on the dual Boolean algebra.  
	\item The operation $f$ is continuous.
\end{enumerate}
\end{corollary}

Proposition~\ref{prop:cont=Kripke} and its corollary allow us to relate extended dual spaces and the standard notion of topological algebras. 

\begin{definition}
Given an operational type $\tau$, a \emph{topological algebra of type $\tau$} is an algebra of type $\tau$ in the category of topological spaces. That is, it is a topological space equipped with an algebraic structure of type $\tau$ for which each basic operation is continuous (in the case of an $n$-ary operation we equip the domain with the product topology). Homomorphisms of topological algebras are maps which are simultaneously homomorphisms for the algebra structure and continuous for the topological structure. Isomorphisms must also be homeomorphisms for the topological part of the structure. A topological algebra is said to be a \emph{Boolean-topological algebra} provided the underlying topological space is a Boolean space, i.e., it is compact Hausdorff with a basis of clopen sets. Finally, a \emph{Priestley topological algebra} is an algebra in the category of Priestley spaces. That is, it is a Priestley space equipped with an algebra structuresuch that each basic operation of the algebra is not only continuous but also order preserving. The homomorphisms are algebra homomorphisms that are continuous and order preserving, whereas isomorphisms also have to be homeomorphisms for the topological structure and isomorphisms for the order structure. 
\end{definition}

Applying the implications (1)$\Rightarrow$(2) and/or (1)$\Rightarrow$(3), and noticing that conditions (2) and (3) are precisely the $n$-ary versions of the conditions for being dual to a residual operation given in Definition~\ref{def:topprores}, we obtain the following corollary of Proposition~\ref{prop:cont=Kripke}.

\begin{corollary}
Every Priestley topological algebra is the dual space of some bounded distributive lattice with additional operations.
\end{corollary}

\begin{corollary}\label{cor:charBootopalg}
The Boolean-topological algebras are precisely the extended Boolean dual spaces of Boolean algebras with residuation operations for which the dual relations are functional.
\end{corollary}

Proposition~\ref{prop:cont=Kripke} establishes a connection between the duals of \emph{residual} operations and continuous maps. One may wonder what it takes for the forward image map to be the dual of an operation. Using the condition, as given in Section~\ref{sec:duality}, on a relation equivalent to it being the dual of the \emph{forward image} operation, we obtain the following requirements in the case of a functional relation on a Boolean space:
\begin{enumerate}
\item For all $x\in X$ the preimage $f^{-1}(x)$ is closed;
\item The forward image of a tuple of clopens is clopen.  
\end{enumerate}
Without continuity this is not a very natural condition for a map between topological spaces. However, we do obtain the following useful corollary.

\begin{corollary}
Let $X$ be a Boolean-topological algebra, $B$ the dual Boolean algebra, and $f$ one of the basic operations of $X$.  Then $f$ is an open mapping if and only if $B$ is closed under the forward operation $(U_1,\ldots,U_n)\mapsto f[U_1\times\ldots\times U_n]$, and in this case the graph $R$ of $f$ is also the relational dual to this forward operation on $B$.
\end{corollary}

\begin{proof}
Since all continuous maps from compact spaces to Hausdorff spaces are closed mappings, it follows that $B$ is closed under the operation $(U_1,\ldots,U_n)\mapsto f[U_1\times\ldots\times U_n]$ if and only if $f$ is an open map. The conditions required for $R$ to be the dual of this operations are $f^{-1}(x)=R[\underline{\ \ \ },x]$ closed for each $x\in X$ and $f[U_1\times\ldots\times U_n]=R[U_1,\ldots,U_n,\underline{\ }]$ clopen whenever the $U_i$s are clopen. For $f$ continuous and $X$ Hausdorff the first condition always holds. And if in addition $f$ is open, then the second also holds.
\end{proof}

In the remainder of this subsection we investigate the relationship between the conditions (1)--(4) of Proposition~\ref{prop:cont=Kripke} further. In particular we show that, for Priestley spaces in general, condition (1) is equivalent to (2) (and (3)) in the case of unary operations, but that (2) and (3) are equivalent to neither of (1) and (4) in general. 

\begin{proposition}\label{prop:unaryfunctrel}
Let $X$ be a Priestley space and $R\subseteq X\times X$ an order-compatible and functional binary relation on $X$ with $f$ the corresponding operation. Then the following conditions are equivalent.
\begin{enumerate}
	   \item The operation $f$ is continuous with respect to the Priestley topology. 
	   \item For all $x \in X$, $R[x,\underline{\ }]$ is closed in $X$ and for all clopen down-sets $V\subseteq X$ the relational 
	            image $R[\_,V^c]$ is clopen.
\end{enumerate}
Furthermore, if these conditions are satisfied, then the dual operation $V\mapsto (R[\_,V^c])^c$ is the distributive lattice homomorphism dual to $f$.
\end{proposition}

\begin{proof}
We already know from Proposition~\ref{prop:cont=Kripke} that (1) implies (2). For the reverse implication, note that if $V$ is a clopen down-set then its complement is an up-set, and for up-sets $U$ we have
\[
R[\_,U]=\{x\in X\mid \exists y\in U\text{ with }f(x)\geq y\}=\{x\in X\mid f(x)\in U\}=f^{-1}[U].
\]
So condition (2) implies that the preimages of clopen up-sets are clopen. Now since $f^{-1}[U^c]=(f^{-1}[U])^c$, the preimages of clopen down-sets are also clopen, and since the clopen up-sets and clopen down-sets together form a subbasis for the Priestley topology, it follows that (2) implies (1) as required. 

Finally, we now see that the operation $V\mapsto (R[\_,V^c])^c$ is equal to the lattice homomorphism $V\mapsto f^{-1}[V]$ since, by the above computation $(R[\_,V^c])^c=(f^{-1}[V^c])^c=((f^{-1}[V])^c)^c=f^{-1}[V]$ for any clopen down-set $V$.
\end{proof}

We thus see that unary Priestley topological algebras are rather trivial, as one also expects since the dual of an order preserving continuous map under Priestley duality is a homomorphism.

\begin{corollary}\label{cor:contunary=dual of endo}
The unary Priestley topological algebras are precisely the extended Priestley dual spaces of distributive lattices with additional operations which are endomorphisms of the lattice. 
\end{corollary}

As detailed in the following example, we can also use Proposition~\ref{prop:unaryfunctrel} to show that the last condition of Proposition~\ref{prop:cont=Kripke} is not equivalent to the first three in general.

\begin{example}
Consider the bounded distributive lattice $D$ of all subsets of $\Z$ which are either finite or $\Z$ itself. 
The Priestley dual of $D$ is the poset $X=\Z\oplus \infty$ obtained by adding $\infty$ as a top to the trivially ordered anti-chain $\Z$. The topology on $X$ is the one of the one-point compactification by $\infty$ of the discrete space on $\Z$.
Now consider the map $f:X\to X$ which sends any $k\in\Z$ to $0$ and $\infty$ to $\infty$. Then $f$ is continuous in the spectral topology but not in the Priestley topology. Conditions (2) and (3) of Proposition~\ref{prop:cont=Kripke} also are not satisfied since, by Proposition~\ref{prop:unaryfunctrel}, these are equivalent to continuity in the Priestley topology in the unary case.
\end{example}

Finally, we also give an example to show that the first condition of Proposition~\ref{prop:cont=Kripke} is not equivalent to the last three in general. 

\begin{example}\label{ex:not unifcont}
Consider again the bounded distributive lattice $D$ of all subsets of $\Z$ which are either finite or $\Z$ itself. We consider the residuals of addition on $\Z$ lifted to the power set of $\Z$. Since addition is commutative, the right and left residuals agree and we need only consider one of them. Note that for $A,B\in{\mathcal P}(\Z)$, we have
\[
A/B=\bigcap\{A/k\mid k\in B\}=\bigcap\{A-k\mid k\in B\}.
\]
Further, it is clear that $D$ is closed under $(\ )-k$ as $F-k=\{n-k\mid n\in F\}$ is again finite for $F$ finite and $\Z-k=\Z$. Also, $D$ is closed under arbitrary intersections, so $D$ is closed under residuation. As in the previous example, the Priestley dual of $D$ is the poset $X=\Z\oplus \infty$ obtained by adding $\infty$ as a top to the trivially ordered anti-chain $\Z$ and the topology on $X$ is the one of the one-point compactification by $\infty$ of the discrete space on $\Z$. It is straight forward to verify that the ternary relation dual to the residuation operation on $D$ is functional with its upper edge given by addition on $X$ defined as usual; for $i,j\in\Z$:
\[
\begin{tabular}{|c|c|c|}
	\hline
	$\,+\,\tvi$ & $i$ & $\infty$  \\
	\hline
	$j\tvi$ & $\,i+j\,$ & $\,\infty\,$   \\
	\hline
	$\infty\tvi$ & $\infty$ & $\infty$ \\
	\hline
\end{tabular} 
\]
Since we are dealing with the extended dual of a residuation operation, conditions (2) and (3), and thus also (4) of Proposition~\ref{prop:cont=Kripke} must be satisfied. However, this addition operation is not continuous in the Priestley topology since, e.g., the singleton $\{0\}$ is clopen but its preimage, $\{(k,-k)\mid k\in \Z\}$, is open but not closed. 
\end{example}

We conclude that in Proposition~\ref{prop:cont=Kripke} conditions (2) and (3) are neither equivalent to condition (1) nor to condition (4) in general. 
We postpone a characterization of the duals of Priestley topological algebras to Section~\ref{subsec:resideals}.

\subsection{Residuation algebras preserving joins at primes}

In this subsection, we characterize the additional operations on lattices for which the extended Priestley dual relations are functional. In the exposition, we will mainly focus on the binary case in order to lighten the notation. The results do go through for higher arities as well though.

For a unary operation on a lattice, we saw in Proposition~\ref{prop:unaryfunctrel}, that its dual relation is functional if and only if the operation is in fact an endomorphism of the lattice. However, the situation is far from this trivial in the binary and higher arity setting. In fact, in arities greater than or equal to two, a dual relation may be functional without the original map preserving both meet and join. To see this, consider the set $X=\{0,1,-1\}$ with the ternary relation given by usual multiplication. The binary residuation operation $/$ on $B=\mathcal P(X)$ preserves meet in its first coordinate and reverses joins in the second, that is, the identities $(A_1\cap A_2)/B=(A_1/B)\cap(A_2/B)$ and $A/(B_1\cup B_2)=(A/B_1)\cap(A/B_2)$ hold in $B$. However, the operation $/$ does not preserve join in the first coordinate nor does it reverse meets in the second, e.g., $\{-1,1\}/\{-1,1\}=\{-1,1\}$ but $\{-1\}/\{-1,1\}=\emptyset=\{1\}/\{-1,1\}$, and $\{1\}/\emptyset=X$ is strictly larger than the union of $\{1\}/\{1\}=\{1\}$ and $\{1\}/\{-1\}=\{-1\}$.

 As we have seen in the previous subsection, the appropriate operations on lattices dual to functional relations are residuation operations. We make the following definition.

\begin{definition}
A \emph{(binary) residuation algebra} is a bounded distributive lattice, $D$, equipped with two binary operations $\bs, /:D\times D\to D$ with the following properties:
\begin{enumerate} 
\item $\forall a,b_1,b_2\in D\qquad\qquad a\bs (b_1\wedge b_2)=(a\bs b_1)\wedge (a\bs b_2).$
\item $\forall a_1,a_2,b\in D\qquad\qquad (a_1\wedge a_2)/b=(a_1/b)\wedge (a_2/b).$
\item The two operations $\bs$ and $/$ are linked by the Galois property:
\[
\hspace*{-3cm}\forall a,b,c\in D\qquad\qquad b\leq a\bs c \iff a\leq c/b.
\]
\end{enumerate}
Note that, under the assumption of (3) conditions (1) and (2) are equivalent so that one holds if and only if the other does.  The Galois property (3) also implies that the following two properties hold:
\begin{enumerate} 
\item[(4)] $\forall a_1,a_2,b\in D\qquad\qquad (a_1\vee a_2)\bs b=(a_1\bs b)\wedge (a_2\bs b).$
\item[(5)] $\forall a,b_1,b_2\in D\qquad\qquad a/(b_1\vee b_2)=(a/b_1)\wedge (a/b_2).$
\end{enumerate}
In general, for an operational type $\tau$, by \emph{residuation algebra of type $\tau$}, we mean a bounded distributive lattice with operations corresponding to the $n$ residuals of an $n$-ary operation for each $n$-ary function symbol in $\tau$. We call such an algebra a \emph{Boolean residuation algebra} provided the underlying lattice is Boolean.
\end{definition}

For simplicity of notation, we will mainly deal with the binary case. Given a binary residuation algebra we have a ternary dual relation $R$ on the dual $X$ of $D$ which encodes both the operations and which is given by
\begin{align*}
R(x,y,z) \iff F_x\bs I_z\leq I_y
                         \iff I_z/ F_y\leq I_x.
\end{align*}
As explained in Section~\ref{ExtDuality}, this relation will be Priestley compatible with respect to the first two coordinates. Even though a residuation algebra $D$ does not in general have an operation $\cdot$ for which $\bs$ and $/$ are the residuals, we do obtain a product operation on the filter lattice of $D$.

\begin{proposition}
Let $D$ be a binary residuation algebra and $(X,\leq,\pi,R)$ its extended Priestley dual. Then we have an operation $\cdot:Filt(D)\times Filt(D)\to Filt(D)$ given by 
\begin{align*}
F\cdot G=\{c\mid\exists a\in F\mbox{ with }a\bs c\in G\}.
\end{align*}
Further, it is related to the liftings of $\,\bs$ and $/$ and to $R$ by the following 
multi-sorted residuation property
\begin{align*}\label{Rdef}
\hspace*{-.1cm}\forall x,y,z\in X \qquad R(x,y,z) \iff F_x\cdot F_y\geq F_z
                         \iff F_x\bs I_z\leq I_y
                         \iff I_z/ F_y\leq I_x.
\end{align*}
In general, given a residuation algebra of type $\tau$, we obtain, for each $n$-ary operation symbol, an $n$-ary operation on the filter lattice.
\end{proposition}

\begin{proof} We just prove the proposition in the binary case. Let $F,G\in Filt(D)$ and $H=\{c\mid\exists a\in F\mbox{ with }a\bs c\in G\}$. Since $\bs$ is order preserving in its second coordinate and $G$ is an up-set, it is clear that $H$ is an up-set. Also as $1\bs 1=1$, $H$ is non-empty. We show that $H$ is closed under binary meet. To this end, let $a_i\in F$ with $a_i\bs c_i\in G$ for $i=1$ and $2$. Then $(a_1\wedge a_2)\bs c_i\geq a_i\bs c_i$ and thus $(a_1\wedge a_2)\bs c_i\in G$ for $i=1$ and $2$. It follows that
\[
[(a_1\wedge a_2)\bs c_1]\wedge[(a_1\wedge a_2)\bs c_2]=(a_1\wedge a_2)\bs (c_1\wedge c_2)\in G.
\] 
Now since $F$ is a filter $a_1\wedge a_2\in F$ and thus $c_1\wedge c_2\in H$ and $F\cdot G$ as given is indeed a filter. For the second assertion, let $x,y,z\in X$. Note that we just need to show that $F_x\cdot F_y\geq F_z$ is equivalent to the three other conditions as they are already known to be equivalent by the basic extended duality results. To this end, suppose $F_x\bs I_z\leq I_y$ and $c$ is such that there exists $a\in F_x$ with $a\bs c\in F_y$. It clearly follows that $c\not\in I_z$ and thus $c\in F_z$ as required. Conversely, suppose $F_x\cdot F_y\geq F_z$, that is, $\{c\mid\exists a\in F_x\mbox{ with }a\bs c\in F_y\}\subseteq F_z$ and let $a\in F_x$ and $c\in I_z$. Then $c\not\in F_z$ and thus $a\bs c\not\in F_y$. That is, $a\bs c\in I_y$ as required. 
\end{proof}

We have just shown the existence of a binary operation on the lattice of all filters of a residuation algebra $(D, \bs,/)$ for which $\bs$ and $/$ are in some sense the residuals. The existence of this operation and its relation to  $\bs$ and $/$ is much simpler to understand from the point of view of canonical extensions. There, one may show that the so-called $\pi$-extensions of $\bs$ and $/$ have a lower adjoint on the canonical extension of $D$. The operation given here on filters is then the restriction of this operation to the filter elements of the canonical extension. For more details, see \cite[Lemma~2.22]{GehrkeJonsson04} and the approach in \cite{Gehrke09}.

\begin{proposition}\label{prop:functdual}
Let $D$ be a binary residuation algebra and $(X,\leq,\pi,R)$ its extended Priestley dual. The following conditions are equivalent:
\begin{enumerate}
\item The operation $\cdot:Filt(D)\times Filt(D)\to Filt(D)$ sends prime filters to prime filters.
\item $\forall a,b,c\in D\ \forall x\in X\ \left(a\in F_x\ \implies\ 
\exists a'\in F_x\ [a\bs(b\vee c)\leq(a'\bs b)\vee(a'\bs c)] \right)$.
\item $\forall x\in X$ the map $F_x\bs(\underline{\ }):Idl(D)\to Idl(D)$ is $\vee$-preserving.
\end{enumerate}
A corresponding result holds for residuation algebras of arbitrary type.
\end{proposition}

\begin{proof}
(1)$\Rightarrow$(2): Let $a,b,c\in D$ and $x\in X$ with $a\in F_x$. Further let $y\in X$ with $a\bs(b\vee c)\in F_y$. Then $b\vee c\in F_x\cdot F_y$ and by (1) it follows that $b\in F_x\cdot F_y$ or $c\in F_x\cdot F_y$. Thus there is $a'\in F_x$ with $a'\bs b\in F_y$ or $a'\bs c\in F_y$. In either case, it follows that $(a'\bs b)\vee(a'\bs c)\in F_y$. Thus we have
\[
\eta_D(a\bs[b\vee c])\subseteq\bigcup_{a'\in F_x}\eta_D(a'\bs b)\cup\eta_D(a'\bs c)
\]
where $\eta_D:D\to\mathcal P(X)$ is the embedding given by Priestley duality. Since $\eta_D(a\bs[b\vee c])$ is compact there are $a_1',\ldots,a_n'\in F_x$ with 
\begin{align*}
\eta_D(a\bs[b\vee c])&\subseteq\eta_D(a_1'\bs b)\cup\eta_D(a_1'\bs c)\cup\ldots\cup\eta_D(a_n'\bs b)
\cup\eta_D(a_n'\bs c)\\
                           &=\eta_D([a''\bs b]\vee[a''\bs c])
\end{align*}
where $a''=a_1'\wedge\ldots\wedge a_n'$. That is, $a\bs[b\vee c]\leq[a''\bs b]\vee[a''\bs c]$ and $a''\in F_x$ as required.

(2)$\Rightarrow$(3): Let $x\in X$ and let $I$ and $J$ be ideals of $D$. It is clear that $F_x\bs I\vee F_x\bs J\subseteq F_x\bs(I\vee J)$ since the operation $F_x\bs(\underline{\ })$ is order preserving. Now let $a\in F_x$ and $d\leq b\vee c$ where $b\in I$ and $c\in J$. Then $a\bs d\leq a\bs(b\vee c)$ and by (2) there is $a'\in F_x$ with $a\bs(b\vee c)\leq(a'\bs b)\vee(a'\bs c)$. Since $a'\bs b\in F_x\bs I$ and $a'\bs c\in F_x\bs J$ it follows that $a\bs d\in F_x\bs I\vee F_x\bs J$ and thus the operation $F_x\bs(\underline{\ })$ is join preserving.

(3)$\Rightarrow$(1): Let $x,y\in X$. We want to show that $F_x\cdot F_y$ is prime. Let $b,c\in D$ and suppose $b\not\in F_x\cdot F_y$ and $c\not\in F_x\cdot F_y$. Note that $b\not\in F_x\cdot F_y$ implies that for all $a\in F_x$ we have $a\bs b\not\in F_y$, that is, $a\bs b\in I_y$ and thus, as $\bs$ is order preserving in the second coordinate, $F_x\bs{\downarrow}b\subseteq I_y$. Similarly, $F_x\bs{\downarrow}c\subseteq I_y$. Now, since $F_x\bs(\underline{\ })$ is join preserving, we have 
$F_x\bs{\downarrow}(b\vee c)=F_x\bs({\downarrow}b\vee{\downarrow}c)=(F_x\bs{\downarrow}b)\vee(F_x\bs{\downarrow}c)$ and thus $F_x\bs{\downarrow}(b\vee c)\subseteq I_y$. That is, $b\vee c\not\in F_x\cdot F_y$ and, by contraposition, we have proved that $F_x\cdot F_y$ is prime.
\end{proof}

Note that the first condition is equivalent to $\bs$ and $/$ sending primes (filter-ideal pairs and ideal-filter pairs, respectively) to prime ideals.

\begin{definition}
Let $D$ be a binary residuation algebra. If the equivalent conditions of Proposition~\ref{prop:functdual} hold, then we say that \emph{residuation is join preserving at primes}. In this case we denote by $\cdot$ also the function from $X\times X$ to $X$ such that $F_{x\cdot y}=F_x\cdot F_y$.

For a residuation algebra of type $\tau$, if the operation on filters sends tuples of primes to primes for each basic operation symbol of the type then the dual space becomes a $\tau$ algebra in these operations restricted to primes and we say that residuation is join preserving at primes. 
\end{definition}

Note that if residuation in a residuation algebra $D$ is join preserving at primes, then the relation $R$ dual to a binary operation symbol $\cdot$ is given by $R(x,y,z) \Leftrightarrow x\cdot y\geq z$. In the Boolean case this means that $R$ is the graph of the operation $\cdot$ on $X$ and in the distributive lattice case that $\cdot$ is the `upper-edge' of $R$. In either case, this is precisely the meaning of $R$ being functional as in Definition~\ref{def:functrel}. Combining Proposition~\ref{prop:cont=Kripke} and Proposition~\ref{prop:functdual}, we now obtain the following result.

\begin{theorem}\label{thrm:CharBootopalg}
Let $\tau$ be an operational type. Boolean-topological algebras of type $\tau$ are, up to isomorphism, 
precisely the extended Stone duals of Boolean residuation algebras of type $\tau$ for which residuation 
is join preserving at primes.\\
\end{theorem}

We do not know whether the property of preserving joins at primes is equivalent to a first-order property of residuation algebras.

\subsection{Duals of Priestley topological algebra morphisms}
\label{subsec:topalgmorphisms}

We have seen that Boolean and Priestley topological algebras are special extended dual spaces. However, the appropriate maps for extended dual spaces are the bounded morphisms as described at the end of Section~\ref{ExtDuality} and the appropriate maps for topological algebras are continuous homomorphisms. We start by observing that, in the Boolean case, a continuous map between Boolean-topological algebras is a homomorphism for a basic operation $f$ if and only if it satisfies the (Forth) condition, as given at the end of Section~\ref{sec:duality}, for any (and then all) of the residuated family of operations associated with the graph $R$ of $f$. We spell out the ensuing result for a binary operation in the following proposition. A corresponding result holds in any arity.
\begin{proposition}\label{prop:dualBootopalgmorphism}
Let $\varphi:X\to Y$ be a continuous map between Boolean-topological algebras of the same type $\tau$ and $h:C\to B$ the Boolean algebra homomorphism dual to $\varphi$. Let $\cdot$ be a basic binary operation symbol for the type $\tau$ and $\bs,/$ the residuation operations dual to $\cdot$. Then the following conditions are equivalent:
\begin{enumerate}
\item For all $x_1,x_2\in X$ we have $\varphi(x_1\cdot x_2)=\varphi(x_1)\cdot\varphi(x_2)$.
\item For all $c_1,c_2\in C$ we have $h(c_1\bs c_2)\leq h(c_1)\bs h(c_2)$.
\item For all $c_1,c_2\in C$ we have $h(c_1/c_2)\leq h(c_1)/h(c_2)$.
\end{enumerate}
\end{proposition}

\begin{proof}
Condition (1) is clearly equivalent to the condition
\begin{itemize}
\item[]\hspace*{-1.25cm}(Forth)\qquad\qquad $(R(x_1,x_2,x_3)\ \Rightarrow\  S(\varphi(x_1),\varphi(x_2),\varphi(x_3))\,)$
\end{itemize}
where $R$ is the graph of the operation denoted by $f$ in $X$ and $S$ is the graph of the operation denoted by $f$ in $Y$.
One may easily verify, see e.g. the proof of \cite[Theorem~2.3.1(1)]{Goldblatt89} or the last part of \cite[Section~5]{Gehrke13}, that the (Forth) condition corresponds, for a dual operator, to the inclusions given in (2) and (3).
\end{proof}

In the distributive setting, the (Forth) condition only tells us that  $\varphi(x_1\cdot x_2)\leq \varphi(x_1)\cdot\varphi(x_2)$, and, even in the Boolean case, the above proposition does not give the complete picture. Of central importance to the work on recognition is to know which Boolean subalgebras of the dual of a Boolean-topological algebra are dual to algebraic quotients and the above result does not help us in answering this question. The following result for the distributive setting is much more useful in this regard.

\begin{theorem}\label{thrm:dualtopalgmorphism}
Let $\varphi:X\to Y$ be a continuous and order preserving map between Priestley topological algebras of the same type $\tau$ and $h:E\to D$ the distributive lattice homomorphism dual to $\varphi$. Let $\cdot$ be a basic binary operation symbol for the type $\tau$ and $\bs,/$ the residuation operations dual to $\cdot$. Then the following conditions are equivalent:
\begin{enumerate}
\item For all $x_1,x_2\in X$ we have $\varphi(x_1\cdot x_2)=\varphi(x_1)\cdot\varphi(x_2)$.
\item For all $a\in D, e\in E$ there is $e'\in E$ with $a\leq h(e')$ and $a\bs h(e)=h(e'\bs e)$.
\item For all $a\in D,e\in E$ there is $e'\in E$ with $a\leq h(e')$ and $h(e)/a=h(e/e')$.
\end{enumerate}
\end{theorem}

\begin{proof}
We just prove that (1) and (2) are equivalent; the proof of the equivalence of (1) and (3) being similar. In order to prove that (1) implies (2), consider $a\in D$ and $e\in E$. Then $a\bs h(e)\in D$ and thus $\eta_D(a\bs h(e))$ is a clopen down-set of $X$.  By duality and since $R(x_1,x_2,x_3)$ if and only if $x_3\leq x_1\cdot x_2$, we have
\begin{align*}
\eta_D(a\bs h(e))&=\left(R[\eta_D(a),\underline{\ \ },(\eta_D(h(e)))^c]\right)^c\\
                      &=\left(R[\eta_D(a),\underline{\ \ },(\varphi^{-1}(\eta_E(e)))^c]\right)^c\\
                      &=\{z\in X\mid\forall z'\ (z'\in\eta_D(a)\implies z'\cdot z\in\varphi^{-1}(\eta_E(e))\}\\ 
                      &=\{z\in X\mid \varphi(\eta_D(a)\cdot\{z\})\subseteq\eta_E(e)\}. 
\end{align*}
Now by (1) it follows that $\varphi(\eta_D(a)\cdot\{z\})=\varphi(\eta_D(a))\cdot\{\varphi(z)\}$ so that
\[
\eta_D(a\bs h(e))=\{z\in X\mid \varphi(\eta_D(a))\cdot\{\varphi(z)\}\subseteq\eta_E(e)\}. 
\]
Also, since $\eta_E(e)$ is a down-set and $\cdot$ is order preserving, we have 
\[
\eta_D(a\bs h(e))=\{z\in X\mid {\downarrow}\varphi(\eta_D(a))\cdot\{\varphi(z)\}\subseteq\eta_E(e)\}. 
\]
Note that, since $\varphi$ is continuous, $X$ is compact, and $Y$ is Hausdorff, it follows that $\varphi(\eta_D(a))$ is closed. Also, since the down-set of a closed set in a Priestley space is again closed \cite[Exercise 11.14(ii)]{DaveyPriestley02}, it follows that ${\downarrow}\varphi(\eta_D(a))$ is a closed down-set in $Y$. But closed down-sets in Priestley spaces are all intersections of clopen down-sets, see \cite[Exercise 11.14(iii)]{DaveyPriestley02}. So  
\[
{\downarrow}\varphi(\eta_D(a))=\bigcap\{\eta_E(e')\mid \varphi(\eta_D(a))\subseteq\eta_E(e')\}
\] 
Now, by the definition of $\varphi$ as the dual of $h$, we have for all $e'\in E$
 \[
 \varphi(\eta_D(a))\subseteq \eta_E(e') \ \iff\ \eta_D(a)\subseteq\varphi^{-1}(\eta_E(e') \ \iff\ a\leq h(e').
 \]
It follows that
\[
{\downarrow}\varphi(\eta_D(a))=\bigcap\{\eta_E(e')\mid a\leq h(e'), e'\in E\}.
\] 

Accordingly the condition ${\downarrow}\varphi(\eta_D(a))\cdot\{\varphi(z)\}\subseteq\eta_E(e)$ is equivalent to 
\[
\left(\bigcap\{\eta_E(e')\mid a\leq h(e'), e'\in E\}\right)\cdot\{\varphi(z)\}\subseteq\eta_E(e),
\]
and since multiplication by $\varphi(z)$ is a continuous function, $X$ is compact, and $Y$ is $T_1$, it is a general topological fact that 
\[
\left(\bigcap\{\eta_E(e')\mid a\leq h(e'), e'\in E\}\right)\cdot\{\varphi(z)\}=\bigcap\{\eta_E(e')\cdot\{\varphi(z)\}\mid a\leq h(e'), e'\in E\}.
\]
By compactness we obtain
\begin{align*}
\eta_D(a\bs h(e))&=\{z\in X\mid \exists e'\in E\text{ with }a\leq h(e')\text{ and }\eta_E(e')\cdot\{\varphi(z)\}\subseteq\eta_E(e)\}\\
                           &=\bigcup\{\eta_D(h(e'\bs e))\mid e'\in E\text{ and }a\leq h(e')\}.
\end{align*}
And finally by compactness again, there is $e'\in E$ with $a\leq h(e')$ such that $\eta_D(a\bs h(e))=\eta_D(h(e'\bs e))$ or equivalently  $a\bs h(e)=h(e'\bs e)$.

For the converse, suppose (2) holds and let $x_1,x_2\in X$. In order to show that $\varphi(x_1\cdot x_2)=\varphi(x_1)\cdot\varphi(x_2)$, it suffices to show that for all $e\in E$ we have $\varphi(x_1\cdot x_2)\in\eta_E(e)$ if and only if $\varphi(x_1)\cdot\varphi(x_2)\in\eta_E(e)$. To this end, using the definition of the product relative to the residuals and using (2), we have the following string of equivalences:
\begin{align*}
\varphi(x_1\cdot x_2)\in\eta_E(e)&\iff h(e)\in F_{x_1}\cdot F_{x_2}\\
                                                    &\iff \exists a\in F_{x_1}\text{ with } a\bs h(e)\in F_{x_2}\\
                                                    &\iff\exists e'\in E\text{ with }h(e')\in F_{x_1}\text{ and }h(e'\bs e)\in F_{x_2}\\
                                                    &\iff\exists e'\in E\text{ with }e'\in h^{-1}(F_{x_1})=F_{\varphi(x_1)}\text{ and }e'\bs e\in F_{\varphi(x_2)}\\
                                                    &\iff e\in F_{\varphi(x_1)}\cdot F_{\varphi(x_2)}\\
                                                    &\iff \varphi(x_1)\cdot\varphi(x_2)\in\eta_E(e)
\end{align*}
\end{proof}

While one can verify the correctness of the above proof as given, we note that the result is much more transparent in the setting of canonical extensions. The following example shows that the homomorphism dual to a Boolean-topological algebra morphism need not preserve the residuation operations. 

\begin{example}\label{ex:dualnotpresres}
Let $A=\{\alpha,\beta\}$ and let $A^*$ be the free monoid over $A$, or equivalently, the set of all words over $A$ with the concatenation product. We denote by $1$ the empty word. Then $\mathcal P(A^*)$ is a Boolean algebra with a full residuated family of binary operations on it given as in discrete duality 
\begin{align*}
K\cdot L&=\{uv\mid u\in K, v\in L\}\\
 K\bs L&=\{u\in A^*\mid K\cdot\{u\}\subseteq L\}\\ 
 L/K&=\{u\in A^*\mid \{u\}\cdot K\subseteq L\}.
\end{align*}
For singleton sets, we will write $u^*$ instead of $\{u\}^*$, we will write $K\cdot L$ as $KL$, and $K\cup L$ as $K+ L$ as is usual in the theory of formal languages. We write $A^+$ for the free semigroup generated by $A$, that is, $A^+=A^*-\{1\}$. Let $C$ be the Boolean subalgebra of $\mathcal P(A^*)$ generated by the two languages $\alpha^*$ and $\beta^*$. Then $C$ is a Boolean residuation subalgebra of $\mathcal P(A^*)$ (though it is not closed under the forward operation $\cdot$). The dual space of $C$ has four elements, which may be identified with the four atoms of $C$, namely $y_1=\{1\}, y_\alpha=\alpha^+$, consisting of all non-empty words in the single letter $\alpha$, $y_\beta=\beta^+$, and $y_0=(\alpha^*+\beta^*)^c$. Also, the relational dual of the residuation on $C$ is functional. In fact, it may be verified directly that the dual of $C$ is the discrete idempotent monoid on $Y=\{y_1,y_\alpha,y_\beta,y_0\}$ in which $y_1$ is the identity element, $y_0$ is absorbent, and $y_\alpha y_\beta=y_\beta y_\alpha=y_0$.

Similarly, let $B$ be the Boolean residuation subalgebra of $\mathcal P(\{\alpha\}^*)$ with atoms $x_1=\{1\}$, and $x_\alpha=\alpha^+$. Here too it may be verified that the dual of $B$ is the discrete idempotent monoid on $X=\{x_1,x_\alpha\}$ in which $x_1$ is the identity element. 

Observe that the map $\varphi:X\to Y$ given by $x_1\mapsto y_1$ and $x_\alpha\mapsto y_\alpha$ is a Boolean-topological monoid morphism. The dual of $\varphi$ is a Boolean algebra homomorphism $h:C\to B$ for which we have 
\[
h(\beta^*\bs \alpha^*)=h(\emptyset)=\emptyset\not=\alpha^*=\{1\}\bs\alpha^*=h(\beta^*)\bs h(\alpha^*).
\]
\end{example}

As we will see in the next subsection though, the dual of a surjective Priestley topological algebra morphism does preserve residuation. This fact, in conjunction with Theorem~\ref{thrm:dualtopalgmorphism} allows us to give a nice dual characterization of Priestley topological algebra quotients. 

\subsection{Residuation ideals and quotients of Boolean-topological algebras}
\label{subsec:resideals}

In this subsection we will identify which distributive sublattices of the dual of a Priestley topological algebra correspond to its Priestley topological algebra quotients. That is, we give a characterization among all sublattices, not just among sublattices with a residuation algebra structure already known to preserve joins at primes. To solve this problem, we will show, first of all, that, in the case of a surjective Priestley topological morphism, the dual map preserves the residuation operations. The main result of this subsection is the duality-theoretic cornerstone of the classical Eilenberg-Reiterman theory.	

\begin{proposition}
Let $X$ and $Y$ be Priestley topological algebras of type $\tau$, and let $D$ and $E$ be the dual residuation algebras, respectively. If $\varphi:X\to Y$ is a surjective morphism of Priestley topological algebras of type $\tau$, then the dual of $\varphi$ embeds $E$ in $D$ as a residuation subalgebra of type $\tau$.
\end{proposition}

\begin{proof}
Let $f$ be a basic operation symbol of the type $\tau$. Assume $f$ is binary. Let $\bs$ and $/$ be the residual operations dual to $f$. We show that the dual $h:E\to D$ of $\varphi$ preserves $\bs$. The proof for $/$ and for operations of higher arity is similar. 

Since $\varphi$ preserves $f$, we are in the situation of Theorem~\ref{thrm:dualtopalgmorphism}. Therefore, as we saw in the proof of that theorem, we have, for $a\in D$ and $e\in E$
\[
\eta_D(a\bs h(e))=\bigcup\{\eta_D(h(e'\bs e))\mid e'\in E\text{ and }a\leq h(e')\}.
\]
Thus, for $e_1,e_2\in E$ we have
\[
\eta_D(h(e_1)\bs h(e_2))=\bigcup\{\eta_D(h(e'\bs e_2))\mid e'\in E\text{ and }h(e_1)\leq h(e')\}.
\]
Since $\varphi$ is surjective, by duality, $h$ is injective, and thus we have $h(e_1)\leq h(e')$ if and only if $e_1\leq e'$, and when $e_1\leq e'$ then $e'\bs e_2\leq e_1\bs e_2$ and thus the collection that we take the union of above has a largest element so that 
\[
\eta_D(h(e_1)\bs h(e_2))=\eta_D(h(e_1\bs e_2))
\]
or equivalently $h(e_1)\bs h(e_2)=h(e_1\bs e_2)$ as required.
\end{proof}

As Example~\ref{ex:dualnotpresres} at the end of the previous subsection showed, the surjectivity is essential in the above proposition. In the surjective setting, we get a dual map which is a residuation algebra morphism. This is a stronger property than the one given in Proposition~\ref{prop:dualBootopalgmorphism} for duals of maps between Boolean-topological algebras. We now give an example showing that being a Boolean residuation subalgebra is not sufficient for the dual to be a Boolean-topological algebra quotient. 

\begin{example}\label{ex:JE's}
Let $B$ be the Boolean residuation subalgebra of $\mathcal P(\alpha^*)$ generated by $L_0=(\alpha^3)^*$, $L_1=(\alpha^3)^*\alpha$, and $L_2=(\alpha^3)^*\alpha^2$. Then the dual of $B$ is based on $X=\{L_0,L_1,L_2\}$ and is  isomorphic to the additive group $\Z/3\Z$ (with the operation as on the subscripts). Now let $C$ be the subalgebra of $B$ with elements $\emptyset, L_0, L_0^c$, and  $\alpha^*$. Then $C$ is closed under residuation but it is not a residuation ideal of $B$ since $L_1\bs L_0=L_2\not\in C$. One can check that the dual of $C$ is the two element discrete space $Y=\{L_0, L_0^c\}$ with the ternary relation

\[
\begin{tabular}{|c||c|c|}
	\hline
	$\,R\,\tvi$ & $L_0$ & $L_0^c$  \\
	\hline
	\vspace*{-.5cm}&&\\
	\hline
	$L_0\tvi$ & $\,L_0\,$ & $\,L_0^c\,$   \\
	\hline
	$L_0^c\tvi$ & $L_0^c$ & $L_0,L_0^c$ \\
	\hline
\end{tabular} 
\]

\bigskip

\noindent That is, $R(L_0^c,L_0^c,L_0)$ and $R(L_0^c,L_0^c,L_0^c)$ so that $R$ is not functional. Note that $C$ is a Boolean residuation subalgebra but the embedding of $C$ in $B$ does not satisfy the equivalent conditions of Theorem~\ref{thrm:dualtopalgmorphism} and the residuation algebra $C$ does not preserves joins at primes. 
\end{example}

\begin{definition}
	Let  $B$ be a residuation algebra. We call a subset $C$ of $B$ a \emph{residuation ideal} of $B$ provided 
	\begin{enumerate}
		\item $C$ is a bounded sublattice of $B$; 
		\item for all $c \in C$ and $b \in B: c / b \in C \text{ and } b \bs c \in C$.
	\end{enumerate}
	A \emph{Boolean residuation ideal} of $B$ is a residuation ideal that is a Boolean algebra.
\end{definition}

As we saw in Section~\ref{PrDuality} bounded sublattices of a bounded distributive lattice correspond to compatible quasiorders on the dual space. Furthermore we saw in Section~\ref{ExtDuality} that the additional operations of a residuation algebra correspond to relations on the dual space. It will turn out that the compatible quasiorders corresponding to residuation ideals are exactly those that are relational congruences for the 
corresponding relations in the following sense.

\begin{definition}\label{def:Rcong}
Let $X$ be an extended Priestley space and $\preceq$ a compatible quasiorder on $X$. We say that $\preceq$ is a \emph{relational congruence} on $X$ provided for each basic $(n+1)$-ary relation $R$ on $X$ and for all $x_1,\ldots,x_n, x'_1, \ldots, x'_n, z \in X$ we have
\[
[x'_1\succeq x_1, \ldots, x'_n\succeq x_n\text{ and } R(x_1,\ldots,x_n, z)]\implies
\exists z'\ [ R(x'_1, \ldots, x'_n,z')\ \text{ and }\ z'\succeq z].
\]
\end{definition} 

The following correspondence theorem, in the setting of functional relations, is the main technical result behind \cite[Theorem~7.2]{GeGrPi08}. In essentially as general a form as given here, it is due to Mirte Dekkers and was first proved in her Master's Thesis \cite{Dekkers08}.

\begin{theorem}\label{Rcong}
Let $B$ be a residuation algebra and $C$ a bounded sublattice of $B$. Furthermore let $X$ be the extended dual space of $B$ and $\preceq$ the compatible quasiorder on $X$ corresponding to $C$. Then $C$ is a residuation ideal of $B$ if and only if the quasiorder $\preceq$ is a relational congruence on $X$. 
\end{theorem}

Note that the setting here is more general than in Theorem~\ref{thrm:dualtopalgmorphism} since we do not assume that residuation preserves joins at primes, that is, that the relations on the space are functional.

\begin{proof}
We just prove the theorem for a single binary operation. Suppose that $C$ is a residuation ideal of $B$ and let $x,x',y,y',z\in X$ with $R(x,y,z)$, $x\preceq x'$, and $y\preceq y'$. Then $F_x\cdot F_y\geq F_z$, $F_{x'}\cap C\subseteq F_x$, $I_x\cap C\subseteq I_{x'}$, and similarly for the $y$'s. Now let $F=F_{x'}\cdot F_{y'}$ and $I={\downarrow}(I_z\cap C)$ where the down-set is taken in $B$. Then $F$ is a filter of $B$ and $I$ is an ideal of $B$. We claim that $F$ and $I$ are disjoint.

To this end, suppose $a\in F$ and $c\in C$ with $a\leq c$. Since $a\in F$ there is $b\in F_{x'}$ with $b\bs a\in F_{y'}$. Now $a\leq c$ implies $b\bs a\leq b\bs c$ and thus $b\bs c\in F_{y'}$. Also, as $C$ is a residuation ideal $b\bs c\in C$ and thus $b\bs c\in F_{y'}\cap C\subseteq F_y$. Since the maps $c/(\ ),(\ )\bs c$ form a Galois connection, we have 
\[
	b\bs c = (c/(b\bs c))\bs c\ \text{ and }\ b\leq c/(b\bs c).
\]
Thus we have $(c/(b\bs c))\bs c\in F_y$ and $c/(b\bs c)\in F_{x'}\cap C\subseteq F_x$ so that $c\in F_x\cdot F_y$. Now since $F_x\cdot F_y\geq F_z$, it follows that $c\in F_z$ and thus $c\not\in I_z$. That is, $F$ and $I$ are disjoint.

Finally, by the Prime Filter Theorem, we obtain $z'\in X$ with $F\subseteq F_{z'}$ and $I\cap F_{z'}=\emptyset$. It follows that $I_z\cap C\subseteq I_{z'}$ and thus $z\preceq z'$. Also $F\subseteq F_{z'}$ so that $F_{x'}\cdot F_{y'}\geq F_{z'}$ and thus $R(x',y',z')$.

For the converse, suppose $\preceq$ is a relational congruence. We want to show that 
\begin{align*}
C&=\{c\in B\mid \forall y,y'\ (y\preceq y'\implies(c\in F_{y'}\implies c\in F_y))\}\\
  &=\{c\in B\mid \forall y,y'\ (y\preceq y'\implies(c\in I_y\implies c\in I_{y'}))\}
\end{align*}
is a residuation ideal of $B$. To this end, let $c\in C$, $b\in B$, and $y,y'\in X$ with $y\preceq y'$ and $b\bs c\in I_y$. Using the fact that $b\bs c=(R[\eta_B(b),\underline{\ },(\eta_B(c))^c])^c$, we see that 
\[
b\bs c\in I_y \ \iff \ \exists x,z\in X\ [R(x,y,z)\text{ and } b\in F_x \text{ and } c\in I_z].
\]
Thus we may pick $x,z\in X$ with $R(x,y,z)$, $b\in F_x$, and $c\in I_z$. Now since $y\preceq y'$ it follows from the fact that $\preceq$ is a relational congruence that there exists $z'\in X$ with $R(x,y',z')$ and $z'\succeq z$. From  $z'\succeq z$ we obtain $I_z\cap C\subseteq I_{z'}$ and thus $c\in I_{z'}$. In all we have $R(x,y',z')$ and $b\in F_x$ and $c\in I_{z'}$ so that $b\bs c\in I_{y'}$ as desired. Similarly we can prove that $c/b \in C$ for all $c \in C, b \in B$.
\end{proof}

Restricting to the setting where the original space has functional relations we see more clearly why relational congruence is the right name for the concept introduced in Definition~\ref{def:Rcong}.

\begin{lemma}\label{lem:funtrelcong}
Let $X$ be a Priestley space, $R$ a compatible functional relation with corresponding operation $f:X^n\to X$. Further let $\preceq$ be a compatible quasiorder on $X$. Then the following conditions are equivalent:
\begin{enumerate}
\item $\preceq$ is a relational congruence for $R$;
\item For all $\overline{x},\overline{x}{\,'}\in X^n\qquad \left(\overline{x}\preceq\overline{x}{\,'}\ \implies\ f(\overline{x})\preceq f(\overline{x}{\,'})\right)$. 
\end{enumerate}
\end{lemma}

\begin{proof}
Assume (1) and let $\overline{x},\overline{x}{\,'}\in X^n$ with $x_i\preceq x_i'$ for each $i\in\{1,\ldots,n\}$. Since $R$ is functional with corresponding operation $f$, we have $R(\overline{x},f(\overline{x}))$. Thus, by (1), there is $z'\in X$ with $f(\overline{x})\preceq z'$ and $R(\overline{x}{\,'},z')$. Therefore, again because $R$ is functional with corresponding operation $f$, it follows that $z'\leq f(\overline{x}{\,'})$. Now, since $\preceq$ extends $\leq$ and is transitive we have $f(\overline{x})\preceq f(\overline{x}{\,'})$ as required.

For the converse, suppose $\overline{x},\overline{x}{\,'}\in X^n$ with $x_i\preceq x_i'$ for each $i\in\{1,\ldots,n\}$ and that $R(\overline{x},z)$. Then $z\leq f(\overline{x})$. Also, by (2), $f(\overline{x})\preceq f(\overline{x}{\,'})$, and thus $z\preceq f(\overline{x}{\,'})$. Let $z'=f(\overline{x}{\,'})$, then $R(\overline{x}{\,'},z')$ since $R$ is functional with corresponding function $f$. Also $z\preceq z'$, so (1) is satisfied.
\end{proof}

Now we just need an order theoretic and a topological generality, respectively, to be able to interpret Theorem~\ref{thrm:dualtopalgmorphism} in terms of quotients of Priestley topological algebras.

\begin{lemma}\label{lem:orderprescong}
Let $X$ be a poset, $\preceq$ be a quasiorder on $X$ extending the order on $X$, and let $\equiv\,=\,\preceq\cap\succeq$ be the equivalence relation corresponding to $\preceq$. Further, let $f:X^n\to X$ be an order preserving operation on $X$. Then the following conditions are equivalent:
\begin{enumerate}
\item For all \ $\overline{x},\overline{x}{\,'}\in X^n\qquad \left(\overline{x}\preceq\overline{x}{\,'}\ \implies\ f(\overline{x})\preceq f(\overline{x}{\,'})\right)$;
\item $\equiv$ \ is a congruence for $f$ and the quotient operation $f/{\equiv}:(X/{\equiv})^n\to X/{\equiv}$ is order preserving. 
\end{enumerate}
\end{lemma}

\begin{proof}
This equivalence is a straight forward verification using the fact that the order on the quotient satisfies
\[
x/{\equiv}\,\leq\,y/{\equiv}\ \iff\ x\preceq y.
\]
\end{proof}

\begin{lemma}\label{lem:contcong}
Let $X$ be a compact space and $f:X^n\to X$ continuous. If $q:X\twoheadrightarrow Y$ is a Hausdorff quotient of $X$ whose kernel is a congruence for $f$, then the quotient operation $f/{\equiv}:Y^n\to Y$ is continuous.
\end{lemma}

\begin{proof}
We have the following diagram:
\begin{center}
\unitlength=25pt 
\begin{picture}(4,2.5)(0,1)\nullfont
	\gasset{Nframe=n}
	\node(X)(1,1){$X$} 
	\node(Xn)(1,3){$X^n$}
	\node(Y)(4,1){$Y$} 
	\node(Yn)(4,3){$Y^n$} 
	\drawedge(Xn,X){$f$} 
	\drawedge(Xn,Yn){$q^n$} 
	\drawedge(X,Y){$q$} 
	\drawedge(Yn,Y){$g$} 
\end{picture}
\end{center}
where $g=f/{\equiv}$ is the operation on the quotient given by $f$. It suffices to show that $g^{-1}(C)$ is closed whenever $C$ is closed in $Y$. For $C$ is closed in $Y$,  the set $f^{-1}(q^{-1}(C))$ is closed in $X^n$ by continuity of $q$ and $f$. Also, since $X^n$ is compact and $Y$ is Hausdorff, $q^n$ is a closed mapping and thus $q^n(f^{-1}(q^{-1}(C)))$ is closed in $Y^n$. However, since $q^n$ is surjective and the diagram commutes, we have $q^n(f^{-1}(q^{-1}(C)))=q^n((q^n)^{-1}(g^{-1}(C)))=g^{-1}(C)$.
\end{proof}

As a consequence we obtain a dual characterization of the quotients of a Priestley topological algebra.

\begin{theorem}\label{thrm:Bootopalgquotients}
Let $X$ be a Priestley topological algebra. Then the Priestley topological algebra quotients of $X$ are in one-to-one correspondence with the residuation ideals of the residuation algebra $B$ dual to $X$. 
\end{theorem}

Apart from being important in applications, this theorem also allows us to characterize the duals of Priestley topological algebras. 

\begin{corollary}
Let $\tau$ be an operational type. Priestley topological algebras of type $\tau$ are, up to isomorphism, precisely the extended Priestley duals of residuation algebras that embed as residuation ideals in Boolean residuation algebras of type $\tau$ for which residuation is join preserving at primes. 
\end{corollary}

\begin{proof}
Suppose the dual of $X$ embeds as a residuation ideal in a Boolean residuation algebra of type $\tau$ for which residuation is join preserving at primes. Since residuation in the Boolean residuation algebra is join preserving at primes, by Theorem~\ref{thrm:CharBootopalg}, its dual, $Y$, is a Boolean-topological algebra of type $\tau$. Furthermore, since the dual of $X$ embeds as residuation ideal in the dual of $Y$, by Theorem~\ref{thrm:Bootopalgquotients} above, $X$ is a Priestley topological algebra quotient of $Y$ and thus in particular it is a Priestley topological algebra.

Conversely, if $X$ is a Priestley topological algebra, then forgetting the order on the space yields a Boolean-topological algebra, $X_-$. Also, since the Priestley order on $X$ is a compatible (quasi)order on $X_-$ (see Remark~\ref{rem:Priestleyspace}), it follows that the identity map $id:X_-\to X$ is a Priestley topological algebra quotient map. By Theorem~\ref{thrm:Bootopalgquotients}, this implies that the dual of $X$ embeds as a residuation ideal in the dual of $X_-$, which, by Theorem~\ref{thrm:CharBootopalg}, is a Boolean residuation algebra for which residuation is join preserving at primes. 
\end{proof}

												            %
\section{Profinite algebras and applications} \label{sec:ProfAlgAppl}      %
												            %

The applications of topological algebra in automata theory and finite model theory, as well as in many parts of classical algebra, are mainly concerned with profinite algebras. In this section we restrict our attention to these. First, we give a characterization of the Boolean residuation algebras dual to profinite algebras.  Next we consider the further special case of profinite completions. We show that the residuation algebra dual to the profinite completion of a (discrete) algebra is the Boolean residuation algebra of recognizable subsets of the original algebra. Finally we show how the generalization of the composition of Eilenberg's and Reiterman's theorem obtained in \cite{GeGrPi08} is a special case of the duality between sublattices of a bounded distributive lattice and quotients of its dual space.  

\subsection{Dual characterization of profinite algebras}       %

In this section we characterize the Boolean residuation algebras dual to profinite algebras. We illustrate the use of this characterization by giving a simple proof of the fact that any Boolean-topological algebra quotient of a profinite algebra is again profinite.

Let $X$ be a topological algebra of type $\tau$. By definition, $X$ is \emph{profinite} provided there is a family $\{X_i\}$ of finite $\tau$-algebras indexed by a directed set $I$ and $\tau$-algebra morphisms $f_{ij}:X_i\to X_j$ for $i\geq j$, so that $f_{jk}\circ f_{ij}=f_{ik}$ whenever $i\geq j\geq k$ and $X=\varprojlim(X_i,f_{ij})$.  If $X$ is profinite, then, by the standard construction of the inverse limit, $X$ embeds in the product of the $X_i$ and the inverse images of subsets of $X_i$ under the inverse limit maps $f_i:X\to X_i$, where $i$ ranges over all elements of $I$, form a basis for the topology of $X$. Thus $X$ is in fact a Boolean-topological algebra. For more details on inverse limits, see \cite[Sections I.4 and I.6]{Bourbaki1}. Since $X$ is a Boolean-topological algebra, by Theorem~\ref{thrm:CharBootopalg}, it follows that $X$ is the extended Stone dual of a Boolean residuation algebra $B$ of type $\tau$. Furthermore, without loss of generality, we may assume that all the maps $f_{ij}$ are surjective and thus that each $X_i$ is a topological quotient algebra of $X$. That is, by  Theorem~\ref{thrm:CharBootopalg} and Theorem~\ref{thrm:Bootopalgquotients}, $X$ is a profinite topological algebra if and only if there is a directed family $\{B_i\}$ of finite Boolean residuation ideals of $B$ such that $B=\bigcup B_i$. But this in turn is clearly equivalent to finitely generated Boolean residuation ideals of $B$ being finite.

\begin{definition}
Let $B$ be a Boolean residuation algebra. We say that $B$ is \emph{locally finite with respect to residuation ideals} provided each finite subset of $B$ generates a finite Boolean residuation ideal in $B$.
\end{definition}

We now have the following characterization of profinite topological algebras among Boolean-topological algebras.

\begin{theorem}
A Boolean-topological algebra is profinite if and only if the dual residuation algebra is locally finite with respect to residuation ideals.
\end{theorem} 

While it is straight forward to see that products and subobjects of profinite algebras are again profinite, the case of quotient is less obvious. However, using the above duality result, closure under Boolean-topological quotient algebras becomes an easy consequence.

\begin{theorem}
Any Boolean-topological algebra quotient of a profinite algebra is again profinite.
\end{theorem} 

\begin{proof}
Let $X$ be a profinite algebra. That is, the dual residuation algebra $B$ is locally finite with respect to residuation ideals. Now let $X\twoheadrightarrow Y$ be a Boolean-topological algebra quotient of $X$. Then $Y$ is the extended dual of some Boolean residuation algebra $A$ and the dual of the quotient map $X\twoheadrightarrow Y$ is an embedding $A\hookrightarrow B$ of $A$ in $B$ as a Boolean residuation ideal of $B$. Since $B$ is locally finite with respect to residuation ideals so is $A$ and thus $Y$ is also profinite.
\end{proof}

\subsection{Profinite completions and recognizable subsets of an algebra}\label{sec:profcomp}  

We now consider a further restricted class of topological algebras, namely, profinite completions of (discrete) abstract algebras.

Let $A$ be an algebra of some type $\tau$. The \emph{profinite completion} of $A$ is by definition the reflection of $A$ in the category of profinite algebras of type $\tau$. Let $Con_\omega(A)$ denote the set of all congruences $\theta$ of $A$ of finite index (i.e. for  which the quotient algebra $A/\theta$ is finite). Note that if $Con_\omega(A)$ is ordered by reverse inclusion then it is directed since the intersection of two congruences of finite index is again a congruence of finite index. Thus we obtain an inverse limit system, ${\mathcal F}_A$, indexed by $Con_\omega(A)$ as follows:  
\begin{enumerate}
\item For each $\theta\in Con_\omega(A)$ we have the finite algebra $A/\theta$;
\item Whenever $\theta\subseteq\psi$ we have a (unique) homomorphism $A/\theta\to A/\psi$ which commutes with the quotient maps $q_\theta:A\to A/\theta$ and $q_\psi:A\to A/\psi$ and thus the maps of the system also commute with each other as required. 
\end{enumerate}

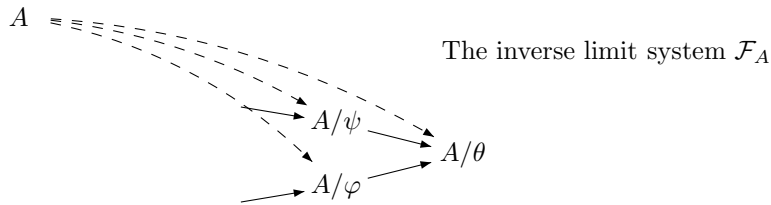
\unitlength=4pt 
\begin{picture}(60,23)(0,-1)\nullfont
	\gasset{Nframe=n,Nw=6,Nh=6,Nmr=2.5,curvedepth=0}
	\put(40,15){$\text{The inverse limit system $\mathcal F_A$}$}
	\node(G)(30,3){$A/\varphi$}
	\node(F)(42,6){$A/\theta$} 
	\node(H)(30,9){$A/\psi$} 
	\node(N1)(18,11){}
	\node(N2)(15,6){}
	\node(N3)(18,1){}
	\node(A)(0,19){$A$}
	\drawedge(H,F){} 
	\drawedge(G,F){} 
	\drawedge(N1,H){} 
	\drawedge(N3,G){} 
	\drawedge[dash={1}0,curvedepth=2](A,H){}
	\drawedge[ELside=r,dash={1}0,curvedepth=3](A,G){}
	\drawedge[ELside=r,dash={1}0,curvedepth=3](A,F){}
\end{picture}

One can show that the inverse limit of this system in the category of topological algebras, $\widehat{A}=\varprojlim{\mathcal F}_A$ is the profinite completion of the algebra $A$. 

Dualising the objects and maps of the system ${\mathcal F}_A$, we obtain a direct limit system $\mathcal G_A$ of finite Boolean residuation subalgebras of type $\tau$ of $\mathcal P(A)$ with maps that are embeddings as residuation ideals. \\

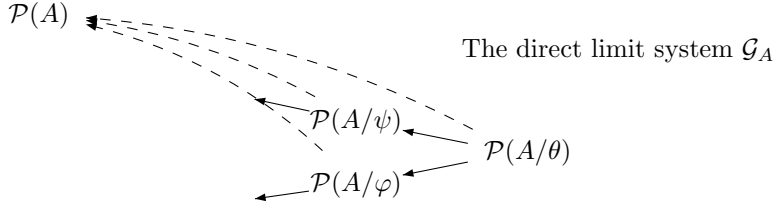
\begin{picture}(60,21)(0,1)\nullfont
	\gasset{Nframe=n,Nw=9,Nh=9,Nmr=9,curvedepth=0}
	\put(40,15){$\text{The direct limit system $\mathcal G_A$}$}
	\node(G)(30,3){\quad$\mathcal P(A/\varphi)$\quad}
	\node(F)(45,6){\quad$\mathcal P(A/\theta)$} 
	\node(H)(30,9){$\mathcal P(A/\psi)$} 
	\node(N1)(16,12){}
	\node(N2)(15,6){}
	\node(N3)(16,1){}
	\node(A)(0,19){$\mathcal P(A)$}
	\drawedge(F,H){} 
	\drawedge(F,G){} 
	\drawedge(H,N1){} 
	\drawedge(G,N3){} 
	\drawedge[dash={1}0,curvedepth=-2](H,A){}
	\drawedge[ELside=r,dash={1}0,curvedepth=-3](G,A){}
	\drawedge[ELside=r,dash={1}0,curvedepth=-3](F,A){}
\end{picture}
\\

The direct limit of this system among Boolean residuation algebras of type $\tau$ is the union of the images of the embeddings $q_\theta^{-1}:\mathcal P(A/\theta)\hookrightarrow\mathcal P(A)$. The Boolean algebra underlying this union is a fundamental and much studied object in theoretical computer science, namely the \emph{Boolean algebra of recognizable subsets} of the algebra $A$. We give the standard definitions: Given a homomorphism $\varphi:A\to F$ into a finite algebra $F$, a subset $L\subseteq A$ is said to be \emph{recognized by $\varphi$} provided there is a subset $P\subseteq F$ with $L=\varphi^{-1}(P)$, or equivalently if $L=\varphi^{-1}(\varphi[L])$. A subset $L\subseteq A$ is said to be \emph{recognized by $F$} provided there is a homomorphism $\varphi:A\to F$ which recognizes $L$, and finally $L$ is said to be \emph{recognizable} provided there is a finite algebra $F$ such that $L$ is recognized by $F$. We denote the Boolean algebra of all recognizable subsets 
of $A$ by $\Rec(A)$. We have: 
\begin{align*}
\Rec(A) &=\{\varphi^{-1}(P)\mid \varphi:A\to F\supseteq P,\varphi\text{ a homomorphism, and }F\text{ finite}\} \\
             &=\bigcup\{q_\theta^{-1}({\mathcal P}(A/\theta))\mid \theta\in Con_\omega(A)\}\\
             &=\varinjlim \{\mathcal P(A/\theta)\}_{\theta\in Con_\omega(A)}.                  
\end{align*}
We would now like to conclude that the extended Stone dual of the Boolean residuation algebra $\Rec(A)$ is the profinite completion $\widehat{A}$. However, at this point the categories we are taking limits in do not quite match across the duality. So either we need to show that residuation in $\Rec(A)$ preserves joins at primes in order to show that it is in the category dual to profinite algebras or we need to show that $\widehat{A}$ is the inverse limit of the system $\mathcal F_A$ in the bigger category of extended Boolean spaces dual to Boolean residuation algebras of type $\tau$. One can verify the former. This is the content of \cite[Proposition 8]{Gehrke09}. Here we opt for the latter as it is conceptually more informative and it is an interesting fact in its own right that inverse limits in the category of extended Boolean Stone duals of any arity are given by the familiar product construction as in profinite algebras. The proof of the following theorem is a bit lengthy but follows the classical style arguments about inverse limits in compact spaces.

\begin{theorem}\label{thrm:extStoneinvlim}
Let $\tau$ be a type of extended Boolean spaces. Inverse limits in the category of extended Boolean spaces of type $\tau$ are given as in the category of topological spaces with the additional relations defined coordinate-wise.
\end{theorem}

\begin{proof}
Let $\{X_i\}_{i\in I}$ be an inverse limit system of extended Boolean spaces of type $\tau$ with corresponding bounded morphisms $f_{ij}: X_i\to X_j$ whenever $i\geq j$. Let $X$ be the inverse limit of the underlying Boolean spaces of the $\{X_i\}_{i\in I}$. That is, $X$ consists of those $\overline{x}\in \Pi_{i\in I}X_i$ such that $f_{ij}(x_i)=x_j$ whenever $i\geq j$. It is straight forward to verify that $X$ is a closed subspace of $\Pi_{i\in I}X_i$ and thus a Boolean space. It is well known to be the inverse limit in the larger category of compact Hausdorff spaces and thus it is the inverse limit of the underlying Boolean spaces of the $\{X_i\}_{i\in I}$ in the category of Boolean spaces. Suppose the type $\tau$ includes a relation symbol $R$ of arity $n$ and that the $m$th coordinate is the codomain coordinate. That is, the spaces $Y$ in the category satisfy 
\begin{enumerate}
	\item For each $y\in Y$ the set $R[\_,y,\_]$ is closed (where $y$ occurs in the $m$th coordinate);
	\item For all $U_1,\ldots,U_{m-1}$ and $V_{m+1},\ldots,V_n$ clopen subsets of $Y$, the set $R[\overline{U},\_,\overline{V}]$ is clopen,
\end{enumerate}
and bounded morphisms satisfy the (Back) condition for the $m$th coordinate as well as the (Forth) condition. See Definition~\ref{def:topprores} and Definition~\ref{def:bddmorphism} where the conditions are given for the more general extended Priestley space setting. Now, we define the relation on $X$ coordinate-wise as in the algebraic setting. That is, for $(\overline{x}_1,\ldots,\overline{x}_n)\in X^n$ we have $R_X(\overline{x}_1,\ldots,\overline{x}_n)$ if and only if $R_i(x_{1i},\ldots,x_{ni})$ for each $i\in I$ where $R_i$ is the interpretation of $R$ in the space $X_i$ and $x_{ki}$ is the $i$th coordinate of $\overline{x}_k$. We need to show that the inverse limit maps $f_i:X\to X_i$, which are just the restriction to $X$ of the  projections $\pi_i:\Pi_{j\in I}X_j\to X_i$  for $i\in I$, are bounded morphisms. For this purpose, fix $i\in I$. We show that $f_i$ is a bounded morphism. Clearly, if $R_X(\overline{x}_1,\ldots,\overline{x}_n)$ then $R_i(x_{1i},\ldots,x_{ni})$ so that the (Forth) condition holds. Now, to ease the notation and without loss of generality, we assume that $m=1$. Suppose further that $R_i(f_i(\overline{x}_{1}),x_{2i},\ldots,x_{ni})$ where $\overline{x}_1\in X$ and $(x_{2i},\ldots,x_{ni})\in X_i^{n-1}$.  We want to show that there exist $(\overline{x}_{2},\ldots,\overline{x}_{n})\in X^{n-1}$ with $R_X(\overline{x}_{1},\overline{x}_{2},\ldots,\overline{x}_{n})$ and $f_i(\overline{x}_k)=x_{ki}$ for each $k\in\{2,\ldots,n\}$. Define for each finite subset $M$ of $I$ containing $i$, the set $S_M\subseteq(\Pi_{j\in I} X_j)^{n-1}$ consisting of those $(\overline{z}_2,\ldots, \overline{z}_n)$ satisfying the following properties
\begin{enumerate}
\item $\pi_i(\overline{z}_k)=x_{ki}$ for each $k\in\{2,\ldots,n\}$;
\item $f_{jj'}(\pi_j(\overline{z}_k))=\pi_{j'}(\overline{z}_k)$ for each $j,j'\in M$ with $j\geq j'$ and $k\in\{2,\ldots,n\}$;
\item $R_j(f_j(\overline{x}_1),\pi_j(\overline{z}_2),\ldots,\pi_j(\overline{z}_n))$ for each $j\in M$.
\end{enumerate}
Since all the functions involved are continuous and the $R_j$s are closed relations, it follows that each $S_M$ is a closed subset of $(\Pi_{j\in I} X_j)^{n-1}$. We show that each $S_M$ is non-empty. Since $M$ is finite and $I$ is directed, there is $l\in I$ with $l\geq j$ for all $j\in M$. Now since $f_{li}:X_l\to X_i$ is a bounded morphism, $f_i(\overline{x}_{1})=f_{li}(f_l(\overline{x}_1))$, and $R_i(f_i(\overline{x}_{1}),x_{2i},\ldots,x_{ni})$, there are $z_{2l},\ldots,z_{nl}\in X_l$ with $R_l(f_l(\overline{x}_1),z_{2l},\ldots,z_{nl})$ and $f_{li}(z_{kl})=x_{ki}$ for each $k\in\{2,\ldots,n\}$. Now any $(\overline{z}_2,\ldots, \overline{z}_n)\in(\Pi_{j\in I} X_j)^{n-1}$ satisfying 
\[
\pi_j(\overline{z}_k)=f_{lj}(z_{kl})
\]
for each $j\in M$ is in $S_M$. So $S_M$ is non-empty. Further, it is clear that if $M\supseteq M'$ then $S_M\subseteq S_M'$ so that the collection $\{S_M\mid i\in M\subseteq I, M\text{ finite}\}$ has the Finite Intersection Property. Thus, by compactness of $(\Pi_{j\in I} X_j)^{n-1}$, the intersection of all the $S_M$'s is non-empty and for any $(\overline{x}_2,\ldots, \overline{x}_n)$ in this intersection, we have by item (2) above that $(\overline{x}_2,\ldots, \overline{x}_n)\in X^{n-1}$. Also, by item (3), we have that $R_j(f_j(\overline{x}_1),f_j(\overline{x}_2),\ldots,f_j(\overline{x}_n))$ for all $j\in I$ and thus $R_X(\overline{x}_1,\overline{x}_2,\ldots,\overline{x}_n)$. Finally, by (1), we have $f_i(\overline{x}_k)=x_{ki}$ for each $k\in\{2,\ldots,n\}$. That is, we have proved that $f_i$ is a bounded morphism.

In order to complete the proof of the fact that $X$ with the $f_{i}$'s is the inverse limit of the given system, we need to show that for any extended Boolean space, $Y$, of type $\tau$ with bounded morphisms $g_i:Y\to X_i$ for $i\in I$ such that $f_{ij}\circ g_i=g_j$ whenever $i\geq j$, the unique continuous map $g:Y\to X$ given by $g(y)=(g_i(y))_{i\in I}$ is a bounded morphism. First note that if $R_Y(y_1,\ldots,y_n)$, then as each $g_i$ is a bounded morphism we have $R_i(g_i(y_1),\ldots,g_i(y_n))$ for each $i\in I$, and thus $R_X(g(y_1),\ldots,g(y_n))$. That is, $g$ satisfies the (Forth) condition. Now, again we assume that $m=1$ to ease notation. Let $y_1\in Y$ and $(\overline{x}_2,\ldots,\overline{x}_n)\in X^{n-1}$ with $R_X(g(y_1),\overline{x}_2,\ldots,\overline{x}_n)$. Let $M\subseteq I$ be finite, then there is $i\in I$ with $i\geq j$ for each $j\in M$. Now since $g_i$ is a bounded morphism there are $(z_2,\ldots,z_n)\in Y^{n-1}$ so that $g_i(z_k)=x_{ki}$ for $2\leq k\leq n$ and $R_Y(y_1,z_2,\ldots,z_n)$. For each $j\in M$ and $k\in\{2,\ldots,n\}$ we have $g_j(z_k)=f_{ij}(g_i(z_k))=f_{ij}(x_{ki})=x_{kj}$. So each set
\begin{align*}
S_M=\{(z_2,\ldots,z_n)&\in Y^{n-1}\mid  R_Y(y_1,z_2,\ldots,z_n)\\
                                    &\text{ and }g_j(z_k)=x_{kj}\text{ whenever } 2\leq k\leq n\text{ and }j\in M\}
\end{align*}
is non-empty. Also, each $S_M$ is closed and the collection of all $S_M$'s for $M\subseteq I$ finite has the Finite Intersection Property. Any  $(y_2,\ldots,y_n)\in\bigcap\{S_M\mid M\subseteq I\text{ finite}\}$ satisfies $R_Y(y_1,y_2,\ldots,y_n)$ and $g(z_k)=\overline{x}_{k}$ for each $k\in\{2,\ldots,n\}$. That is, we have proved that $g$ is a bounded morphism as required.
\end{proof}

The following theorem now follows from the elementary fact that the systems $\mathcal F_A$ and $\mathcal G_A$ are dual to each other under extended Stone duality. 

\begin{theorem}\label{thrm:profcomp}
Let $A$ be an abstract algebra. The profinite completion $\widehat{A}$ is homeomorphic as a topological algebra to the extended Stone dual of $\Rec(A)$, the Boolean algebra with residuation operations of recognizable subsets of $A$.
\end{theorem}

The above theorem, in the case of a finitely generated free monoid, is one of the main results of \cite{GeGrPi08}, cf. Theorem~6.1. Profinite methods have been studied extensively in connection with automata theory including the connection between recognizable subsets and profinite completions as algebras of implicit operations \cite{Almeida94}. The fact that the Boolean algebra of recognizable sets is dual in basic Stone duality to the underlying Stone space of the profinite completion is well known \cite{Almeida94} and was used explicitly by Pippenger in \cite{Pippenger97}. However, the methods of Pippenger's paper were not adopted by others in the area. Most applications using profinite completions of algebras make very essential use of the algebraic operations so that capturing these is essential for the duality to be useful in situations where profinite completions are applied and this is new to the work \cite{GeGrPi08} for which the present paper gives the more complete duality theoretic point of view. The work \cite[Section~8.4]{RhoStei08} is similar in spirit to our work here in that it captures profinite semigroups by dual structures. However, it exploits the connection between Boolean spaces and Boolean rings rather than Boolean lattices and thus goes in the direction of ring theory rather than lattice theory. Another difference is that in the approach of Rhodes and Steinberg, the dual structure is a bi-algebra rather than an algebra: An algebraic operation $f:X^n\to X$ on a dual space most directly is dualized in Stone duality as a Boolean algebra or lattice homomorphism $h: B\to \bigoplus_{i=1}^n B_i$ into the coproduct of $n$ copies of the dual lattice $B$. This is naturally co-algebraic structure on $B$ rather than algebraic structure on $B$. The purpose of the current paper is to explore the less obvious fact that continuous algebraic operations on dual spaces can actually be captured by purely algebraic structure on the dual Boolean algebras. However, exploring the co-algebraic approach in the setting of Stone and Priestley duality seems an interesting direction for further work.
 
\subsection{Profinite term operations}                 %

As is fundamental in universal algebra, the elements of a free algebra over a finite set\footnote{We use $A$ rather than $X$ or $V$ in line with traditions in automata and languages where an alphabet $A$ plays the role of the set of variables.} $A$ of variables may be seen as terms yielding $|A|$-ary term functions on all algebras of the appropriate type. The interest of the profinite completion of such a free algebra is that its elements may be seen as generalized terms yielding so-called profinite $|A|$-ary term functions on all finite algebras of the appropriate type. This is usually shown using uniform continuity. Here we explain this phenomenon from a duality theoretic point of view that does not appeal to uniform continuity but uses a \emph{double dual} construction consisting in applying the extended discrete duality of Section~\ref{subsec:discdual} first and then the extended Stone duality (after appropriate restriction of the codomain).

 Consider a fixed algebra type $\tau$ and let $\mathcal V$ be a variety (i.e., an equational class) of algebras of type $\tau$. Let $A$ be a finite alphabet, and $F_{\mathcal V}(A)$, the free $\mathcal V$-algebra freely generated by $A$. 
 As we saw in Theorem~\ref{thrm:profcomp} above, the Boolean residuation algebra $\Rec(F_{\mathcal V}(A))$ is the extended Stone dual of the profinite completion $\widehat{F_{\mathcal V}(A)}$ of $F_{\mathcal V}(A)$. We want to show that, given a finite algebra $B$ of type $\tau$, every $x\in \widehat{F_{\mathcal V}(A)}$ yields a  function $x^B:\quad\ B^A\  \longrightarrow\ B$, which is then the profinite term function on $B$ induced by $x$. That is, for every finite $\mathcal V$-algebra, $B$, we want to define an assignment
\begin{align*}
(\_)^B:\widehat{F_{\mathcal V}(A)}\ &\longrightarrow\ {\mathcal F}(B^A,B)\\
                    x\  &\ \mapsto\ \ x^B:\ \ B^A\  \longrightarrow\ B\\
                       &\qquad\qquad\quad  (b_a)_{a\in A} \mapsto  x^B((b_a)_{a\in A}),               
\end{align*}
where ${\mathcal F}(B^A,B)$ is the set of functions from $B^A$ to $B$, that is, the set of $|A|$-ary operations on the finite algebra $B$. To this end, note that each tuple $(b_a)_{a\in A}\in B^A$ is a function 
\begin{align*}
\varphi:A & \longrightarrow B\\
 a&\ \mapsto\ b_a. 
\end{align*}
By freeness of $F_{\mathcal V}(A)$ it has a unique extension to a homomorphism 
\[
F(\varphi):F_{\mathcal V}(A)\to B.
\]
Consider the corresponding surjective homomorphism $F(\varphi):F_{\mathcal V}(A)\to B'$ where $B'=Im(F(\varphi))$. The dual of this map under the discrete duality is a complete Boolean algebra embedding that embeds ${\mathcal P}(B')$ as a residuation ideal in ${\mathcal P}(F_{\mathcal V}(A))$
\[
(F(\varphi))^{-1}:{\mathcal P}(B')\hookrightarrow{\mathcal P}(F_{\mathcal V}(A)).
\]
However, by the definition of recognizable subset, the image of this map falls entirely within $\Rec(F_{\mathcal V}(A))$.  That is, $(F(\varphi))^{-1}:{\mathcal P}(B')\hookrightarrow\Rec(F_{\mathcal V}(A))$. Denote the Stone duality functor from the category of Boolean algebras to the category of Boolean spaces by $\mathcal S$. Applying it we obtain 
\[
{\mathcal S}((F(\varphi))^{-1}):\widehat{F_{\mathcal V}(A)}\to {\mathcal S}({\mathcal P}(B'))
\]
and by Theorem~\ref{thrm:Bootopalgquotients} we conclude that this map
is a topological algebra quotient. Finally, since $B$ is finite so is $B'$ and the discrete and the Stone dualities agree and ${\mathcal S}({\mathcal P}(B'))$ is, up to natural isomorphism, just $B'$ so that 
\[
{\mathcal S}((F(\varphi))^{-1}):\widehat{F_{\mathcal V}(A)}\to B'\hookrightarrow B
\]
is the (unique by density) topological algebra homomorphism extending $F(\varphi)$. We define the \emph{term-function associated to} $x\in  \widehat{F_{\mathcal V}(A)}$ to be the $|A|$-ary operation given by:
\begin{align*}
x^B:B^A\  &\longrightarrow\ B\\
       \varphi\ \ &\ \mapsto \ \ {\mathcal S}((F(\varphi))^{-1})(x).               
\end{align*}

\begin{remark}\label{rem:uniqueext}
Once the map $\varphi$ into a finite algebra $B$ as above is extended to a 
homomorphism from the free algebra to $B$, the rest of the extension works for arbitrary algebras. That is, if $A$ now is an algebra, rather than a generating set for one, and $\varphi:A\to B$ is a homomorphism, rather than just a set map, then taking $B'=Im(\varphi)$, the map $\varphi^{-1}:\mathcal P(B')\to\Rec(A)$ is an embedding of a Boolean residuation ideal and thus the extended Stone dual map $\mathcal S(\varphi^{-1}):\widehat{A}\to B'\hookrightarrow B$ is a topological algebra map extending $\varphi$ in the sense that $\varphi=\mathcal S(\varphi^{-1})\circ e$ where $e:A\to\widehat{A}$ is the canonical injection. We will denote this (unique) topological algebra map extending $\varphi:A\to B$ by $\widehat{\varphi}:\widehat{A}\to B$.
\end{remark}

\subsection{Sublattices and equational theories}\label{sec:eqthry}  

In universal algebra, Birkhoff's variety theorem states that classes of algebras closed under homomorphic images, subalgebras, and products are precisely those that are model classes of equational theories. In finite model theory, Reiterman's theorem \cite{Reiterman82} does the same for classes of finite algebras: the classes of finite algebras closed under homomorphic images, subalgebras, and finite products are precisely those that are model classes of profinite equational theories. In the setting of monoids, Eilenberg's theorem relates certain classes of recognizable languages with classes of finite monoids closed under homomorphic images, subalgebras, and finite products. The Eilenberg-Reiterman combination thus relates certain classes of recognizable languages with profinite equational theories. This combination is a central tool in automata theory, where it is often used to obtain the decidability of classes of recognizable languages. For this reason, there are generalizations in various directions that relax one or more of the requirements on the classes of languages to which the theory applies, e.g. \cite{Pin95,PinWeil96b,Pippenger97,Polak05,Esik02,Straubing02,Kunc03}. 

In \cite{GeGrPi08}, it was shown that the Eilenberg-Reiterman combination is in fact a special instance of the Stone duality between sublattices and quotient spaces, thus providing a common generalization in which only closure under intersection and union is required for the classes of languages and duality for other closure properties accounts for the many earlier generalizations of Eilenberg-Reiterman theorems in a completely modular manner\footnote{The theorem of \cite{Kunc03} does not follow as it only requires a semilattice (along with other requirements). Encompassing this result would require generalizing \cite{GeGrPi08}  using duality for distributive meet or join semilattices \cite{BezJan11}.}. In particular, the direct duality route from lattices of languages to (profinite) equational theories is available also when the classes of finite algebras in the middle are not. The relationship between the three types of theorems may be illustrated by the following diagram.
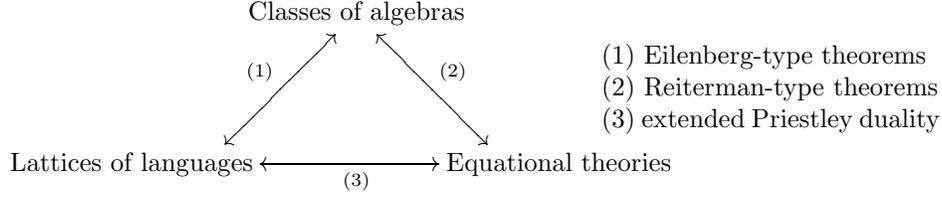
\begin{displaymath}
\hskip1.8cm
\begin{array}{lr}
{\xy
(0,20)*+{\mbox{\mathstrut Classes of algebras}}="t0"; 
(-20,0)*+{\mbox{ \hskip-2.1cm Lattices of languages}}="t1"; 
(20,0)*+{\mbox{Equational theories \hskip -1.5cm }}="b1"; 
{\ar@{<->}_{ (1)}  "t0"; "t1"}; 
{\ar@{<->}_{(2)}"b1"; "t0"}; 
{\ar@{<->}_{ (3)} "t1"; "b1"}; 
\endxy}
&
\text{\parbox{4.5cm}{\vskip-2cm (1) Eilenberg-type theorems\\ (2) Reiterman-type  theorems \\ (3) extended Priestley duality}}
\end{array} 
\end{displaymath}

In this section we start from Theorem~\ref{prop:subalg duality} applied in the special case of a Boolean algebra of recognizable subset of an abstract algebra $A$, and then we specialize in a modular way down to the case of the (composition) of the classical Eilenberg and Reiterman theorems. We fix an algebraic type $\tau$ and, for now, also an algebra $A$ of this type. We are interested in $\Rec(A)$, the lattice (or actually Boolean algebra) of recognizable subsets of $A$ and its sublattices. A fundamental part of duality is that the subobjects of an object on one side of a duality correspond to the quotient objects of the dual object. As we have seen in Theorem~\ref{thrm:profcomp}, the dual space of $\Rec(A)$, as a Boolean 
residuation algebra, is the profinite completion, $\widehat{A}$, as a topological algebra. In a first tempo, we forget the algebraic structure on both of these objects, and we simply have a Boolean algebra and its dual Stone space. Applying Theorem~\ref{prop:subalg duality}, we obtain the following theorem:

\begin{theorem}\label{thrm:eqthry0}
Let $A$ be an abstract algebra, $\Rec(A)$ the Boolean algebra of its recognizable 
subsets, and $\widehat{A}$ its profinite completion. The assignments
\[
\Sigma\mapsto {\mathcal C}_\Sigma=\{L\in \Rec(A)\mid\forall(x,y)\in \Sigma\ 
(L\in F_y\ \Rightarrow\ L\in F_x)\}
\]
for $\Sigma\subseteq \widehat{A}\times\widehat{A}$ and
\[
{\mathcal K}\mapsto \preceq_{\mathcal K}=\{(x,y)\in \widehat{A}\times\widehat{A}
\mid\forall L\in\mathcal K\ (L\in F_y\ \Rightarrow\ L\in F_x)\}
\]
for ${\mathcal K}\subseteq \Rec(A)$ establish a Galois connection whose Galois closed sets are the compatible quasiorders on $\widehat{A}$ and the bounded sublattices of $\Rec(A)$, respectively.
\end{theorem}

Here we want to understand a pair $(x,y)\in \widehat{A}\times\widehat{A}$ as a kind of equation. At this most general level, our concept of equation is more akin to a \emph{relation between generators}. We make the following definition.

\begin{definition}
Let $A$ be an abstract algebra. A \emph{profinite lattice equation} for $A$ is given by a pair of elements $x,y\in\widehat{A}$ and is denoted by $x\to y$. The equation, $x\to y$, is said to be satisfied by $L\in\Rec(A)$ if and only if $L\in F_y$ implies $L\in F_x$. That is, $x\to y$ is satisfied by $L$ if and only if $L\in {\mathcal C}_{\{(x,y)\}}$ as defined by the Galois connection in Theorem~\ref{thrm:eqthry0}.
\end{definition}

With this nomenclature, we see that the Galois connection in Theorem~\ref{thrm:eqthry0} is that between model classes and theories. That is,
a set  $\Sigma\subseteq \widehat{A}\times\widehat{A}$ is a set of equations and 
${\mathcal C}_\Sigma$ is the set of all models of $\Sigma$ while a set 
${\mathcal K}\subseteq \Rec(A)$ is a set of models and $\preceq_{\mathcal K}$
is the theory of ${\mathcal K}$. Further, the fact that the Galois closed sets of 
recognizable sets are exactly the lattices of recognizable sets becomes the following general Eilenberg-Reiterman theorem.

\begin{corollary}\label{cor:E-Rsublat}
	Let $A$ be any algebra. A collection of recognizable subsets of $A$ is a sublattice of $\Rec(A)$ if and only if it can be defined by a set of profinite lattice equations for $A$. 
\end{corollary}

Noting that Boolean subalgebras of $\Rec(A)$ are exactly those for which the 
corresponding compatible quasiorder is an equivalence relation and writing $x \leftrightarrow y$ for the conjunction $x \rightarrow y$ and $y\rightarrow x$, we get an equational description of the Boolean subalgebras of recognizable subsets. We call such $x \leftrightarrow y$ \emph{profinite symmetric lattice equations}. 

\begin{corollary}\label{cor:E-RsubBA}
	Let $A$ be any algebra. A collection of recognizable subsets of $A$ is a Boolean subalgebra of $\Rec(A)$ if and only if it can be defined by a set of symmetric lattice equations for $A$. 
\end{corollary}

The difference between the lattice case and the Boolean case is that we need an order relation in the lattice setting as in Priestley duality. This fact was rediscovered in the theory of formal languages and automata by Pin who introduced ordered monoids and an asymmetric notion of profinite identities \cite{Pin95} without realizing the connection with Priestley duality. 

In the original Eilenberg theorem, not only is it necessary that the collections of recognizable sets be closed under Boolean complementation, they must also be  residuation ideals and be `closed under inverse images of morphisms' (see  Definition~\ref{def:closinvmorph} below). We now proceed to give Eilenberg-Reiterman theorems for each of these conditions separately. 

In order to treat residuation ideals, we need to recall the concept of linear unary polynomial from universal algebra. We will apply it to the algebraic reduct of $\widehat{A}$ for $A$ an algebra of type $\tau$. By definition, the unary polynomials over $\widehat{A}$ are the terms in one variable of the type obtained by expanding $\tau$ with a nullary operation symbol for each element of $\widehat{A}$. The new symbols are then interpreted as themselves in $\widehat{A}$. Furthermore, a unary polynomial  over $\widehat{A}$ is said to be \emph{linear} provided the variable occurs exactly once in the term. We will use the symbol {\small $\Box$} for the variable, and denote the set of all unary linear polynomials over $\widehat{A}$ by $Pol^{\widehat{A}}_{lin}$({\small $\Box$}). Each $p\in Pol^{\widehat{A}}_{lin}$({\small $\Box$}) yields a unary polynomial function $p^{\widehat{A}}:\widehat{A}\to\widehat{A}$ as usual with terms in universal algebra. The linear unary polynomials over an algebra are also sometimes referred to as `contexts' or `terms with a hole'.

\begin{definition}
Let $A$ be an abstract algebra. A \emph{profinite algebra equation} for $A$ is given by a pair of elements $x,y\in\widehat{A}$ and is denoted by $x\leq y$. An equation, $x\leq y$, is said to be satisfied by $L\in\Rec(A)$ if and only if $p^{\widehat{A}}(x)\to p^{\widehat{A}}(y)$ holds for every $p\in Pol^{\widehat{A}}_{lin}$({\small $\Box$}).
\end{definition}

\begin{theorem}\label{thrm:E-Rsubresidl}
	Let $A$ be any algebra. A collection of recognizable subsets of $A$ is a residuation ideal of $\Rec(A)$ if and only if it can be defined by a set of profinite algebra equations for $A$. 
\end{theorem}

\begin{proof}
If $\mathcal C$ is a residuation ideal of $\Rec(A)$, then by Theorem~\ref{Rcong} the corresponding compatible quasiorder $\preceq$ is a congruence on $\widehat{A}$. But by Lemma~\ref{lem:funtrelcong}, since the dual relations on $\widehat{A}$ are functional, we have for each basic operation $f$ of arity $n$ and  for all $\overline{x},\overline{x}{\,'}\in \widehat{A}^n$
\[
\hskip4cm  \overline{x}\preceq\overline{x}{\,'}\ \implies\ f(\overline{x})\preceq f(\overline{x}{\,'}).\hskip3.75cm (*)\\
 \]
 Now, take as set $\Sigma$ of profinite algebra equation all the equations $x\leq y$ such that $x\preceq y$. Then it is clear that if each equation in $\Sigma$ holds for $L$, then $L\in\mathcal C$ since already each $x\to y$ for $x\leq y\in\Sigma$ holding in $L$ implies $L\in\mathcal C$ by Theorem~\ref{thrm:eqthry0}. Conversely, if $x\preceq y$ and $f$ is an $n$-ary basic operation,  $1\leq i\leq n$, and $u_1,\ldots,u_{i-1},v_{i+1},\ldots, v_n\in\widehat{A}$, then $(\overline{u},x,\overline{v})\preceq (\overline{u},y,\overline{v})$ and thus $f(\overline{u},x,\overline{v})\preceq f(\overline{u},y,\overline{v})$ by (*). Since the unary polynomials are built up inductively by applying the basic operations in this manner, we see that (*) implies that for each 
$p\in Pol^{\widehat{A}}_{lin}$({\small $\Box$})
\[
x\preceq y \implies p^{\widehat{A}}(x)\preceq p^{\widehat{A}}(y)\\
\]
That is, for $x$ and $y$ with $x\preceq y$, the set of all $x'\to y'$ corresponding to $x\leq y$ is contained in the set of pairs in $\preceq$. It follows that $x\leq y$  holds for all $L\in \mathcal C$. That is, if $\mathcal C$ is a residuation ideal in $\Rec(A)$ then it is defined by the set $\Sigma=\{x\leq y\mid x\preceq y\}$.
 
For the converse, suppose $L\in \Rec(A)$ satisfies the profinite algebra equation $x\leq y$. We want to show that for any basic operation $g$ of arity $m$ of the type and any $j$ with $1\leq j\leq m$, the $j$th residual $g_j^\sharp$ applied to $L$ in the numerator coordinate and an arbitrary $\overline{K}\in \Rec(A)^{m-1}$ in the denominator coordinates yields an element of $\Rec(A)$ which satisfies $x\leq y$. We just argue in the case of a binary $g$ and $m=1$. The proof for general $g$ and $m$ is similar but much more cumbersome with respect to notation. Let $p\in Pol^{\widehat{A}}_{lin}$({\small $\Box$}), then we want to show that $K\bs L$ satisfies $p^{\widehat{A}}(x)\to p^{\widehat{A}}(y)$. So suppose $K\bs L\in F_{p^{\widehat{A}}(y)}$. By duality this is the same as $p^{\widehat{A}}(y)\in \eta(K\bs L)$ where $\eta=\eta_{\Rec(A)}:\Rec(A)\hookrightarrow\mathcal P(\widehat{A})$ is the Stone embedding. But on the dual space $\widehat{A}$, the residuation operations are given by the graph $R$ of $g$. That is, $\eta(K\bs L)=(R[\eta(K),\underline{\ },(\eta(L)^c]^c$ and this operation on $\mathcal P(\widehat{A})$ has the lifting of $g$ to subsets as lower adjoint, so
\[
\{p^{\widehat{A}}(y)\}\subseteq (R[\eta(K),\underline{\ },(\eta(L)^c]^c\iff g(\eta(K),\{p^{\widehat{A}}(y)\})\subseteq \eta(L).
\] 
That is, for each $w\in \eta(K)$, we have 
\[
g(w,p^{\widehat{A}}(y))\in \eta(L).
\]
Now, for each $w\in\widehat{A}$, the term with constants from $\widehat{A}$ given by $q_w(\text{{\small $\Box$}})=g(w,p(\text{{\small $\Box$}}))$ is again a linear unary polynomial over $\widehat{A}$, and we have $q_w^{\widehat{A}}(y)\in \eta(L)$ for each $w\in \eta(K)$. Since $L$ satisfies $x\leq y$  and $q_w^{\widehat{A}}(y)\in \eta(L)$ is equivalent to $L\in F_{q_w^{\widehat{A}}(y)}$, it follows that $L\in F_{q_w^{\widehat{A}}(x)}$ for each $w\in \eta(K)$. Going backwards in the equivalences used above for $y$, we then obtain $g(\eta(K),\{p^{\widehat{A}}(x)\})\subseteq \eta(L)$ and finally $p^{\widehat{A}}(x)\in \eta(K\bs L)$ or $K\bs L\in F_{p^{\widehat{A}}(x)}$. That is, we have shown that $K\bs L$ satisfies $p^{\widehat{A}}(x)\to p^{\widehat{A}}(y)$ as required.
 \end{proof}

\begin{remark}
These profinite algebra equations are easier to describe in the case of monoids, see \cite{GeGrPi08}. There, for $A$ a monoid and $x,y\in\widehat{A}$, the equation  $x\leq y$ was taken to mean that $uxv\to uyv$ holds for all $u,v\in\widehat{A}$. This is because each unary linear polynomial over a monoid is equivalent to one of the form $u\text{{\small $\Box$}}v$ for some $u,v\in\widehat{A}$. 
\end{remark}

So far our equations are `local' in the sense that they are not invariant under substitution. The last ingredient of the original Reiterman theorem is this invariance. For this purpose we place ourselves within some fixed variety $\mathcal V$ of abstract algebras. A \emph{class of recognizable sets} for $\mathcal V$ is an assignment $A\mapsto \mathcal C(A)$ for each finite alphabet $A$, where $\mathcal C(A)\subseteq \Rec(\FVA)$. We call such a class a \emph{lattice class} provided $\mathcal C(A)$ is a sublattice of $\Rec(\FVA)$ for each finite alphabet $A$. Furthermore, a \emph{class of equations} for $\mathcal V$ is an assignment $A\mapsto \Sigma(A)$ for each finite alphabet $A$, where $\Sigma(A)\subseteq \widehat{\FVA}\times\widehat{\FVA}$. We say that a class $\mathcal C$ is given by a class of equations $\Sigma$ provided, for each finite alphabet $A$, we have that $\mathcal C(A)$ is given by $\Sigma(A)$. Thus Corollary~\ref{cor:E-Rsublat} tells us that a class of recognizable sets is a lattice class if and only if it is given by some class of equations.  

Notice that given  finite alphabets $A$ and $B$ and a homomorphism $\sigma:\FVA\to\FVB$, any recognizable subset $L$ of $\FVB$ has an inverse image under $\sigma$ which is a recognizable subset of $\FVA$, where the recognizing morphism is the pre-composition by $\sigma$ of the recognizing homomorphism for $L$. That is, $\sigma$ induces a Boolean algebra homomorphism
\[
\Rec(\sigma):\Rec(\FVB)\to\Rec(\FVA), L\mapsto \sigma^{-1}(L).
\]
The Stone dual of this homomorphism is a continuous function
\[
\widehat{\sigma}:\widehat{\FVA}\to\widehat{\FVB}.
\]
Since it extends $\sigma$, it is in fact also the unique continuous extension of $\sigma$.

\begin{definition}\label{def:closinvmorph}
A lattice class $\mathcal C$ of recognizable sets is said to be \emph{closed under inverse images of morphisms} provided, whenever $A$ and $B$ are finite alphabets and $\sigma:\FVA\to\FVB$ is a homomorphism, then $L\in\mathcal C(B)$ implies $\sigma^{-1}(L)\in\mathcal C(A)$.

A class $\Sigma$ of equations is said to be \emph{closed under substitution} provided, whenever $A$ and $B$ are finite alphabets and $\sigma:\FVA\to\FVB$ is a homomorphism, then $x\to y\in\Sigma(A)$ implies  $\widehat{\sigma}(x)\to \widehat{\sigma}(y)\in\Sigma(B)$ .
\end{definition}

\begin{theorem}\label{thrm:E-Rinvmorph}
Let $\mathcal C$ be a lattice class of recognizable subsets of a variety $\mathcal V$. Then $\mathcal C$ is closed under inverse images of morphisms if and only if it is given by some equational class which is closed under substitution. 
\end{theorem}

\begin{proof}
Before we prove the theorem, it is worthwhile isolating the following fact: If $\sigma:\FVA\to\FVB$ is a homomorphism, then $\Rec(\sigma):\Rec(\FVB)\to\Rec(\FVA)$ is a Boolean algebra homomorphism and $\widehat{\sigma}:\widehat{\FVA}\to\widehat{\FVB}$ is the dual continuous function. Now for $L\in \Rec(\FVB)$ and $x\in\widehat{\FVA}$ we have
\begin{align*}
L\in F_{\widehat{\sigma}(x)}&\iff \widehat{\sigma}(x)\in \eta_{\Rec(\FVB)}(L)\\
&\iff x\in \widehat{\sigma}^{-1}(\eta_{\Rec(\FVB)}(L))=\eta_{\Rec(\FVB)}(\Rec(\sigma)(L))\\
& \iff \Rec(\sigma)(L)\in F_x.
\end{align*}
Suppose $\mathcal C$ is a lattice class of recognizable sets closed under inverse images of morphisms. Since $\mathcal C$ is a lattice class, it is given by the equational class $\Sigma:A\mapsto \preceq_{\mathcal C(A)}$ as defined in Theorem~\ref{thrm:eqthry0}. Let $A$ and $B$ be finite alphabets and $\sigma:\FVA\to\FVB$ a homomorphism. Since $\mathcal C$ is closed under inverse morphisms, $\Rec(\sigma)(\mathcal C(B))\subseteq\mathcal C(A)$. Let $x\to y$ be an equation in $\Sigma(A)=\preceq_{\mathcal C(A)}$ and $L\in\mathcal C(B)$, then $\Rec(\sigma)(L)$ satisfies $x\to y$. We want to show that $L$ satisfies $\widehat{\sigma}(x)\to\widehat{\sigma}(y)$. To this end assume that  $L\in F_{\widehat{\sigma}(y)}$. Then by the observation above we have $\Rec(\sigma)(L)\in F_y$. Since $\Rec(\sigma)(L)$ satisfies $x\to y$ it follows that $\Rec(\sigma)(L)\in F_x$ and thus, again by the observation above, $L\in F_{\widehat{\sigma}(x)}$. That is, $L$ satisfies $\widehat{\sigma}(x)\to\widehat{\sigma}(y)$ as required.
\end{proof}

Closure under the lattice operations, Boolean complement, residuation, and inverses of morphisms are the hypotheses of the original Eilenberg theorem. As mentioned earlier, various generalizations have allowed the relaxation of certain of these hypotheses while keeping others. The treatment in \cite{GeGrPi08}, for which the duality theoretic components have been given above, is the first fully modular treatment and the first to allow the treatment of lattices of recognizable languages without any further properties. We summarize the results and the location of their proofs in the following table.
\\

\hskip-.3cm
\begin{tabular}{|c|c|c|}
	\hline
	\textbf{Class closed under} & \textbf{Equations} & \textbf{Result} \\
	\hline
	$\cup, \cap$ & $u \rightarrow v$ & Theorem~\ref{thrm:eqthry0} (Corollary~\ref{cor:E-Rsublat}) \\
	\hline
	complement  & $u \leftrightarrow v$ & Theorem~\ref{thrm:eqthry0} (Corollary~\ref{cor:E-RsubBA})  \\
	\hline
	residuation ideal & $u \leq v$ & Theorem~\ref{thrm:E-Rsubresidl}  \\
	\hline
	  inverses of morphisms& substitution invariant & Theorem~\ref{thrm:E-Rinvmorph}\\
	\hline
\end{tabular}
\vskip.5cm
A full account of the method ensuing from these results will be treated elsewhere, see also \cite{pin09} and \cite{Pin2012}. Some applications have already appeared in the literature, see e.g. \cite{BraPin09} and \cite{KuLa12}. A further consequence of the relationship between Eilenberg-Reiterman theory and extended Stone duality is that it allows the application of duality also for classes of languages outside the recognizable fragment. This is the subject of the paper \cite{GehrkeGrigorieffPin10} and of an on-going investigation into language classes given by logic fragments with arbitrary natural number predicates by Gehrke, Krebs, and Pin. This work has also influenced a number of other related works such as \cite{BezKupPan12,BBHPRS,SelKon11,B-BPiS-E12} and it would be interesting to explore the connections of our work with \cite{RhoStei08}.
\bibliography{Duality}
\bibliographystyle{biblistanglais}
\end{document}